\documentclass{CustomArticle}

\usepackage{booktabs}

\usepackage[]{algpseudocode}
\usepackage{algorithm}

\usepackage{natbib}

\usepackage{mathdots}

\usepackage{paralist}
\usepackage{enumitem}
\usepackage[framemethod=tikz]{mdframed}
\usepackage[todo=false,compact=false]{QuickCommands}

\providecommand{\DMsp}{\algname{DM-SP}}
\providecommand{\DMvip}{\algname{DM-VIP}}
\providecommand{\ROM}{\algname{ROM}}
\providecommand{\FDS}{\algname{FDS}}
\providecommand{\ARM}{\algname{ARM}}
\providecommand{\FEGM}{\algname{FEGM}}
\providecommand{\MRN}{\text{\upshape MRN}}
\providecommand{\MS}{\text{\upshape MS}}
\providecommand{\EG}{\algname{EG}}

\newcounter{Block-VIP-assumptions}
\newcounter{VIP-assumptions}
\newcounter{SP-assumptions}

\PaperGeneralInfo{
  title={Efficient Gradient Methods for Distributed Saddle Problems},
  date={May 15, 2026},
}
\AddPaperAuthor{
  name={Ruichen Luo},
  affiliation={Institute of Science and Technology Austria},
  email={rluo@ist.ac.at}
}
\AddPaperAuthor{
  name={Anton Rodomanov},
  affiliation={CISPA Helmholtz Center},
  email={anton.rodomanov@cispa.de}
}
\AddPaperAuthor{
  name={Sebastian U. Stich},
  affiliation={CISPA Helmholtz Center},
  email={stich@cispa.de}
}
\PaperAbstract{  The distributed setting for Saddle Problems~(SPs) has recently emerged as a framework for various modern applications in machine learning and multiagent systems.
  Despite its relevance, the theoretical foundations of this setting have not yet been thoroughly established.
  In this paper, we advance this research direction by formalizing the distributed setup for SPs and providing rigorous definitions of communication and computational costs.
  Our main result is a novel decoupled method that achieves optimal communication cost within the zero-respecting framework.
  Our method is based on a multi-stage reduction to the decoupled minimization of residual norms, which yields strict improvements over the best known communication cost for the class and the long-standing oracle cost of the Extragradient method.
  Further, we show by a matching lower bound that our method is communication-optimal within the family of gradient-span algorithms.
  Finally, we study the extension of distributed SP into Variational Inequality Problem~(VIP), which generalizes two-player zero-sum games to multiplayer general-sum games. We show that our decoupled method achieves a new state-of-the-art communication complexity for this broader class.}
\PaperKeywords{
    convex optimization,
    distributed optimization,
    complexity bounds,
    saddle problems,
    variational inequalities,
    extra-gradient method,
    fast gradient method.}
\PaperThanks{The authors thank Ali Zindari and Krishnendu Chatterjee for their helpful discussions and suggestions on this paper.
RL acknowledges the support of ERC CoG 863818 (ForM-SMArt) and Austrian Science Fund (FWF) 10.55776/COE12. Gemini Pro 3.1 was used for part of the writing and analysis.}

\begin{document}

\PrintTitleAndAbstract

\section{Introduction}
\label{sec:introduction}

\paragraph{Motivation.}

Saddle problems~(SPs) and their generalizations, variational inequality problems~(VIPs),
are of fundamental importance to optimization and game theory.
These problems have a wide array of modern applications, including the training of GANs~\citep{goodfellow2014generative}, robust optimization~\citep{ben2002robust}, and equilibrium computation in game theory and multiagent systems~\citep{von1947theory,rosen1965existence,hu1998multiagent}.

The growing scale of modern problems---driven by applications in machine learning, complex game dynamics, and multiagent protocols---renders reliance on a single central processor increasingly impractical.
Beyond scalability, a more crucial factor is that many applications are inherently \emph{distributed}: agents are often geographically dispersed, driven by their own individual interests, and bound by privacy constraints that prohibit the sharing of raw data or utilities.
Consequently, distributed computation has become an essential regime for these problems.
This perspective underlies a growing body of work in large-scale learning, game-theoretic models, and multiagent systems~\citep{mcmahan2017communication, zhangcommunication2024, conitzer2004communication, nisan2006communication, hart2010long, yoon2025multiplayer}.

In this work, we consider a natural setup where the decision variables and oracles of the SPs or VIPs are partitioned among distributed agents.
For instance, in a classic saddle problem~$\min_{\bfx} \max_{\bfy} F(\bfx, \bfy)$, we consider one agent controls the minimizing variable~$\bfx$ while another controls the maximizing variable~$\bfy$.
This partition naturally models, for instance, the strategic autonomy of players in game theory, the interaction protocol in multiagent systems, and the physical separation of the generator and discriminator in GANs~\citep{conitzer2004communication,goodfellow2014generative}. Since the decision variables are coupled within their utilities, these agents must coordinate to reach a mutual equilibrium. To do so, they form a communication network that allows them to exchange certain information, such as their current decision variables.
Thus, this provides a natural distributed setup where the decision variables are separated among the different agents.

While distributed optimization is well-established for finite-sum minimization and federated learning~\citep{schmidt2017minimizing,mcmahan2017communication}, the literature on SPs has primarily focused on extending these data-distributed paradigms~\citep{deng2021local,beznosikov2025distributed}. In contrast, the study of distributed variables and oracles, which are essential to the multi-agent systems, remains a relatively new topic.

Although a few recent works have touched upon this direction~\citep{zhangcommunication2024,zindari2025decoupled,yoon2025multiplayer,yoon2025pearl}, they predominantly focus on algorithms tailored to specific, favorable scenarios. Consequently, a fundamental gap persists: the lack of a theoretical framework for the general distributed settings of SPs and VIPs.
Existing discussions regarding performance often remain at a vague conceptual level, lacking a rigorous formalization of the distributed environment itself.
Specifically, there are no standardized definitions for communication and oracle costs in this context.
Without such a foundation, it can be difficult to determine the performance limits (lower bounds) or to formally compare the efficiency of different protocols.

To enable a rigorous analysis, it is essential to establish metrics that reflect the constraints of distributed multiagent systems, where network latency and bandwidth often dwarf local processing time. In this regime, the primary bottleneck is the communication cost (exchange rounds), while the computational cost (local gradient queries) is a secondary objective. Viewed through this lens, the Extragradient (\EG) method~\citep{tseng1995linear,nemirovski2004prox} serves as the ``gold standard'' baseline, though the challenges associated with it differ by metric. Regarding computational complexity, consistently improving upon \EG for general monotone problems has remained an elusive goal despite over two decades of research. Regarding communication complexity---a metric that has recently come into focus with the rise of distributed systems---\EG similarly defines the current state-of-the-art. Surpassing this baseline in the general setting represents a new but critical open problem.

This leads to the following research questions:
\begin{itemize}
\item \textbf{Formalization and Limits:} How can we rigorously formalize the communication and oracle costs for distributed SPs?
\item \textbf{Communication Efficiency:} Can we design an algorithm that surpasses the state-of-the-art communication bounds for distributed SPs and VIPs?
\item \textbf{Oracle Efficiency:} Is it possible to consistently improve upon the long-standing oracle complexity of the \EG method for general SPs?
\end{itemize}

\paragraph{Contributions.}

We answer the aforementioned questions in the affirmative, which advances the current theory of distributed SPs and VIPs.
\begin{itemize}
    \item In \cref{sec:convex-concave-composite-saddle-point-problems}, we formalize the distributed saddle-point problem, distributed methods, and their communication and oracle costs. We review \EG and other algorithms, casting them as distributed methods to analyze their costs.
    \item In \cref{sec:DM-for-SP}, we start with a template \DMsp algorithm with a simple, one-loop communication protocol, which improves the state-of-the-art communication cost.
    \item Continuing in \cref{sec:multi-stage-reduction-for-distributed-SPs,sec:implementable-algorithm-SP}, after making the novel multi-stage reduction, we equip the template method with a concrete implementation, thereby consistently improving the (long-standing) oracle cost of \EG for general SPs.
    \item In \cref{sec:lower-complexity-bounds-for-distributed-SP}, by drawing connection to classic convex minimization, we establish the lower bounds for both communication and oracle costs for distributed gradient-span algorithms. In particular, our result shows that our \DMsp algorithm is communication-optimal within the gradient-span algorithm family.
    \item Finally, in \cref{sec:monotone-block-VIPs}, we extend the results to multi-agent settings. We propose \DMvip and improve the state-of-the-art communication cost for the class of distributed VIPs.
\end{itemize}

%


\begin{table}[htbp]
    \centering
    \begin{minipage}{\textwidth}
        \centering
        \caption{Summary of algorithms and complexity results for distributed SPs.}
        \label{tab:communication-complexity}
        \footnotesize 
        \setlength{\tabcolsep}{4pt}
        \renewcommand{\arraystretch}{1.3}
        \resizebox{\textwidth}{!}{%
            \begin{tabular}{llcc}
                \toprule
                \textbf{Method} $\sM$ & \textbf{Communication Cost} & \textbf{Better Oracle?}\textsuperscript{a} & \textbf{Multi-Agent}\textsuperscript{b} \\
                \midrule
                EG\textsuperscript{c} &
                $\cO \bigl( \theta \frac{L_{xy} D_x D_y}{\epsilon} + \frac{L_x D_x^2 + L_y D_y^2}{\epsilon} \bigr)$ & No & Yes \\

                DGDA\textsuperscript{c} &
                $\cO \bigl( \log \frac{1}{\epsilon} \bigr)$ \quad (weakly coupled only) & No & Yes \\

                Cat-EG\textsuperscript{c} &
                ${\cO} \Bigl( \Bigl( \frac{L_{\max} \hat D_x \hat D_y}{\epsilon} + \sqrt { \frac {L_x \hat D_x^2 + L_y \hat D_y^2} {\epsilon} } \Bigr) \log ^2 \bigl( \frac{1}{\epsilon} \bigr) \Bigr)$ & Maybe & No \\

                Cat-Cat-DAGDA\textsuperscript{c} &
                ${\cO} \bigl( \frac{ L_{xy} \hat D_x \hat D_y}{\epsilon} \log^3 \bigl( \frac{1}{\epsilon} \bigr) \bigr)$ & Maybe & No \\
                \midrule

                \textbf{Lower Bound} (Thm.~\labelcref{theorem:lower_complexity-bounds}) &
                $\varOmega \bigl( \frac{L_{xy} D_x D_y}{\epsilon} \bigr)$ & -- & -- \\

                \textbf{DM-SP} (Thm.~\labelcref{thm:DROMsp-communication-complexity}) &
                $\cO \bigl( \theta \frac{L_{xy} D_x D_y}{\epsilon} \bigr)$ & Yes & Yes \\
                \bottomrule
                \multicolumn{4}{l}{%
                    \renewcommand{\arraystretch}{1.0}
                    \begin{tabular}{@{}l@{}}
                        \scriptsize \textsuperscript{a} Indicates whether the method's theoretical oracle cost outperforms the Extragradient~(EG) baseline. \\
                        \scriptsize
                        \textsuperscript{b} Indicates whether the method supports multi-agent extensions. \qquad \textsuperscript{c} These methods are for the non-composite subclass $\sP_{\text{\upshape SP}}^\circ$.
                    \end{tabular}%
                }
            \end{tabular}%
        }
    \end{minipage}
\end{table}

\paragraph{Notations.}
Let $[n] \triangleq \{1, \dots, n \}$, for any positive integer $n$.
For any finite-dimensional real vector space $\cE$, we denote its Euclidean norm by $\Norm{\cdot}_{\cE}$ and its dual norm by $\Norm{\cdot}_{\cE^*}$.
Specifically, we equip the space $\cE_x = \R^{n_x}$ with the norm $\Norm{\bfx}_x = \InnerProduct{\bfP_x \bfx}{\bfx}^{1/2}$, where $\bfP_x \colon \cE_x \to \cE_x^*$ is a self-adjoint positive definite operator and the dual pairing $\InnerProduct{\phi_x}{\bfx}$ denotes $\phi_x(\bfx)$. We denote its corresponding dual norm by $\Norm{\cdot}_{x^*}$.
We assume analogous geometries for $\cE_y = \R^{n_y}$, $\cE_i = \R^{n_i}$ ($i \in [K]$), and $\cE_w = \R^{n_w}$, associated with their respective operators $\bfP_y$, $\bfP_i$, and $\bfP_w$.
For a function $\psi \colon \cE \to \R \cup \{+\infty\}$, let $\dom \psi$ denote its effective domain and $\partial \psi(\bfz)$ its subdifferential at $\bfz \in \dom \psi$.
Finally, for any set of vectors $S$, let $\spn S$ denote its linear span.

\section{Saddle problems with distributed oracles}
\label{sec:convex-concave-composite-saddle-point-problems}

In the context of saddle problems, we consider two separate finite-dimensional real vector spaces, $\cE_x = (\R ^{n_x}, \Norm {\cdot}_x)$ and $\cE_y = (\R ^{n_y}, \Norm {\cdot}_y)$.
We are interested in solving composite Saddle Problems~(SPs) of the following form:
\begin{equation}
\label{eq:saddle-point-problem}
\min _{\bfx \in \dom \psi_x} \ \max _{\bfy \in \dom \psi_y} \ \bigl[ F( \bfx, \bfy ) \triangleq f( \bfx, \bfy ) + \psi_x( \bfx ) - \psi_y( \bfy ) \bigr] ,
\end{equation}
where $\psi_x \colon \cE_x \to \R \cup \{ + \infty \}$ and $\psi_y \colon \cE_y \to \R \cup \{ + \infty \}$ represent relatively simple local components (such as regularizers or indicator functions for constrained sets), and $f(\cdot, \cdot)$ is a real-valued coupling function defined on an open set containing the domain $Q \triangleq \dom \psi_x \times \dom \psi_y$. To simplify the notation, we denote the joint variable by $\bfz \triangleq (\bfx, \bfy) \in Q$.

\subsection{Distributed Methods}
\label{sec:distributed-gradient-based-algorithms}

A method is designed to solve a class of problems sharing a common structure. We begin by introducing the general notion of a problem class, and then formalize what it means for a method to be distributed.

\paragraph{Problem class.}

A \emph{problem class}, denoted by $\sP (\sF, \sO_x, \sO_y, \Delta, \epsilon)$, is a family of problems sharing
\begin{itemize}[nosep,leftmargin=*]
    \item a function family $\sF$ from which each problem instance is drawn;
    \item two distributed oracles $\sO_x$ and $\sO_y$, providing first-order information about an instance;
    \item an accuracy measure $\Delta$, quantifying the quality of a candidate solution; and
    \item a target accuracy $\epsilon > 0$.
\end{itemize}
A specific problem instance $P \in \sP$ is determined by a function instance $\cF \in \sF$, the corresponding oracles $(\sO_x^f, \sO_y^f)$, the accuracy measure $\Delta$, and the target accuracy $\epsilon$. Solving $P$ requires an algorithm to output a candidate solution $\bar \bfz$ satisfying $\Delta(\bar \bfz) \le \epsilon$, accessing only the distributed oracles. The specific problem class of interest in this paper is constructed in \cref{sec:problem-class}.

The rest of this subsection formalizes what it means for an algorithm to be a \emph{distributed method} for solving such a problem class, starting with an engineering description and then a mathematical formalization.

\paragraph{Engineering description.}

We consider a distributed setup with two computational agents, Agent~$x$ and Agent~$y$, each maintaining its decision variable ($\bfx \in \dom \psi_x$ and $\bfy \in \dom \psi_y$, respectively) in its own local memory. The two agents are fully distributed: neither has direct access to the other's memory, and they communicate by exchanging messages over a shared channel. Agent~$x$ has local oracle access to its component $\psi_x$ (for example, the ability to evaluate its proximal-point mapping) and to a problem-dependent first-order oracle $\sO_x$ whose concrete form is specified by the problem class; Agent~$y$ has the analogous access to $\psi_y$ and $\sO_y$. The agents proceed in discrete \emph{communication rounds}. Within each round, each agent performs several local computational steps, each consisting of an oracle query at a chosen point, and the two agents exchange messages at the round boundary. After some number of rounds, the algorithm outputs a candidate solution $\bar \bfz = (\bar \bfx, \bar \bfy)$ that approximately solves Problem~\eqref{eq:saddle-point-problem}.

\paragraph{Mathematical formalization.}

We adopt the framework of \emph{information-based complexity}~\citep{nemirovskij1983problem}. To keep the presentation general, we describe the algorithm in terms of abstract distributed oracles $\sO_x$ and $\sO_y$; their concrete instantiation as partial-gradient oracles for SPs is given in \cref{sec:problem-class}.

Suppose an algorithm $\sM$ proceeds in $T$ communication rounds, where $T$ may be chosen adaptively. In each round $t \in \{0, \dots, T-1\}$, Agent~$x$ successively queries $\sO_x$ at $\tau_x^t$ points $\bfz_x^{t, l}$ for $l = 0, \dots, \tau_x^t - 1$; symmetrically, Agent~$y$ queries $\sO_y$ at $\tau_y^t$ points $\bfz_y^{t, l}$. The number of local steps and the choice of query points are both decided by the agent based on its accumulated information.

We track this via \emph{information sets}, modeled as ordered sequences. Let $I_x^{t, l}$ denote the oracle responses collected by Agent~$x$ prior to its $(l{+}1)$-th query in round $t$. The information set is initialized empty, $I_x^{0, 0} = \emptyset$, and each local query appends the corresponding oracle response, so that $I_x^{t, \tau_x^t}$ summarizes the round; Agent~$y$ is symmetric. Each agent's local data, namely the initial point $\bfz^0$ and its component $\psi_x$ (or $\psi_y$), is treated as known a priori and is not part of the accumulating information set. At the round boundary, the agents exchange messages. From an engineering standpoint, each message is a deterministic function of the sender's accumulated local information. Mathematically, no restriction is imposed on how a message is formed, and so without loss of generality we let each agent read the union of both agents' information at the start of the next round:
\begin{equation}\label{eq:info-merge}
I_x^{t+1, 0} = I_y^{t+1, 0} = \bigl( I_x^{t, \tau_x^t}, I_y^{t, \tau_y^t} \bigr) .
\end{equation}
At the conclusion of round $t$, the algorithm produces a candidate solution $\bar \bfz^{t+1} = (\bar \bfx^{t+1}, \bar \bfy^{t+1})$ from the merged information set $I_x^{t+1, 0} = I_y^{t+1, 0}$.

\begin{definition}
\label{def:distributed-gradient-based}
An algorithm $\sM$ is called a \emph{distributed method} if, for every round $t \in \{0, \dots, T-1\}$, the following hold:
\begin{enumerate}[nosep, leftmargin=*]
    \item Each local query point $\bfz_x^{t, l}$, $l \in \{0, \dots, \tau_x^t - 1\}$, is a deterministic function of Agent~$x$'s current information set $I_x^{t, l}$; symmetrically, each $\bfz_y^{t, l}$, $l \in \{0, \dots, \tau_y^t - 1\}$, is a deterministic function of $I_y^{t, l}$.
    \item The candidate solution $\bar \bfz^{t+1}$ is a deterministic function of the merged information set $I_x^{t+1, 0} = I_y^{t+1, 0}$ defined in~\eqref{eq:info-merge}.
\end{enumerate}
\end{definition}

While these deterministic mappings could in principle be randomized, we restrict attention to the deterministic case here for simplicity.

\paragraph{Communication and oracle complexities.}

For a given problem instance $P$ and target accuracy $\epsilon > 0$, the \emph{communication complexity} of $\sM$ on $P$, denoted $T^\sM_P$, is the smallest integer $k \in \{1, \dots, T\}$ such that the candidate solution $\bar \bfz^k$ satisfies the target accuracy. The total numbers of local oracle queries made by Agent~$x$ and Agent~$y$ up to that point are
\[
N_{x, P}^\sM = \sum_{t=0}^{T^\sM_P - 1} \tau_x^t
\quad \text{and} \quad
N_{y, P}^\sM = \sum_{t=0}^{T^\sM_P - 1} \tau_y^t .
\]
The \emph{communication complexity} and \emph{oracle complexity} of $\sM$ over a problem class $\sP$ are defined by taking the supremum over all instances:
\[
T^\sM_{\sP} = \sup_{P \in \sP} T^\sM_P
\quad \text{and} \quad
N_{\sP}^\sM = \sup_{P \in \sP} \bigl( c_x N_{x, P}^\sM + c_y N_{y, P}^\sM \bigr) ,
\]
where $c_x$ and $c_y$ are fixed constants reflecting the computational costs of evaluating a single query to $\sO_x$ and $\sO_y$, respectively. Because network communication typically forms the main bottleneck in distributed environments, we treat the communication complexity as the primary performance metric and the oracle complexity as a secondary measure of local computational effort.

\subsection{Problem class \texorpdfstring{$\sP_{\text{\upshape SP}}$}{P\_SP}}\label{sec:problem-class}

We now specify the problem class of interest in this paper: composite saddle problems with distributed partial-gradient oracles. The four components of a problem class introduced in \cref{sec:distributed-gradient-based-algorithms} are instantiated in turn below.

\paragraph{Function family.}

We consider the function instances satisfying the following assumptions:
\begin{enumerate}[label={\textnormal{(A\arabic*)}}, nosep, leftmargin=*]
    \item \label{item:continuous-and-convex-concave} For any fixed $\bfy \in \dom \psi_y$, the function $f ( \cdot, \bfy )$ is convex; and for any fixed $\bfx \in \dom \psi_x$, the function $f (\bfx, \cdot)$ is concave. The functions $\psi_x$ and $\psi_y$ are proper, closed, and convex.

    \item \label{item:bounded-distance-to-saddle-point}
    Let $D_x, D_y > 0$ be distance parameters, and let $\bfz^0 = (\bfx^0, \bfy^0) \in Q$ be a given initial point. Relative to this initialization, Problem~\eqref{eq:saddle-point-problem} has a saddle point $(\bfx^*, \bfy^*) \in Q$ such that
    \[
    \lVert \bfx^0 - \bfx^* \rVert_x \le D_x
    \qquad \text{and} \qquad
    \lVert \bfy^0 - \bfy^* \rVert_y \le D_y .
    \]

    \item \label{item:function-smooth} The function $f (\cdot, \cdot)$ is continuously differentiable over $Q$. Moreover, its gradients are Lipschitz continuous. That is, with Lipschitz parameters $L_x, L_{xy}, L_y > 0$, for all $\bfx, \bfx^\prime \in \dom \psi_x$ and $\bfy, \bfy^\prime \in \dom \psi_y$, we have:
    \begin{align*}
    \lVert \nabla_x f( \bfx^\prime, \bfy^\prime ) - \nabla_x f( \bfx, \bfy ) \rVert_{x^*} &\le L_x \lVert \bfx^\prime - \bfx \rVert_x + L_{xy} \lVert \bfy^\prime - \bfy \rVert_y , \\
    \lVert \nabla_y f( \bfx^\prime, \bfy^\prime ) - \nabla_y f( \bfx, \bfy ) \rVert_{y^*} &\le L_{xy} \lVert \bfx^\prime - \bfx \rVert_x + L_y \lVert \bfy^\prime - \bfy \rVert_y .
    \end{align*}
    \setcounter{SP-assumptions}{\value{enumi}}
\end{enumerate}

Let $\sF$ denote the \emph{function family} consisting of all instances $\cF = (f, \psi_x, \psi_y, \bfz^0)$ that satisfy \labelcref{item:continuous-and-convex-concave,item:bounded-distance-to-saddle-point,item:function-smooth} for a fixed set of parameters $(L_x, L_{xy}, L_y, D_x, D_y)$. The non-composite case $\psi_x \equiv \psi_y \equiv 0$ is included as a special instance.

\paragraph{Distributed partial-gradient oracle.}

For SPs, the abstract oracles $\sO_x$ and $\sO_y$ are concretely realized as \emph{(deterministic) partial-gradient oracles}: for a given function instance $\cF \in \sF$ with coupling function $f$ and any input point $\bfz \in Q$,
\begin{itemize}
    \item Agent~$x$ queries $\sO_x$, which returns $\sO_x^f(\bfz) = \nabla_x f(\bfz)$.
    \item Agent~$y$ queries $\sO_y$, which returns $\sO_y^f(\bfz) = \nabla_y f(\bfz)$.
\end{itemize}
The oracles are strictly decoupled: each agent queries only its own oracle, with no access to the counterpart's.

As a concrete example, consider the objective $f(\bfz) = g(\bfz) + f_x(\bfx) - f_y(\bfy)$, where $g(\bfz)$ is a coupled global utility, while $f_x(\bfx)$ and $f_y(\bfy)$ are private utilities accessible only to Agents~$x$ and~$y$, respectively. The partial-gradient oracles then take the form
\[
\sO_x^f(\bfz) = \nabla_x g(\bfz) + \nabla f_x(\bfx) \qquad \sO_y^f(\bfz) = \nabla_y g(\bfz) - \nabla f_y(\bfy) ,
\]
for all $\bfz \in Q$. Due to the distributed setting, Agent~$x$ is entirely blind to the private utility $f_y$ and can only execute $\sO_x$, and vice versa.

\paragraph{Accuracy measure.}

To evaluate the quality of a candidate solution $(\bar \bfx, \bar \bfy) \in Q$, we rely on the restricted duality gap. Let $\cB_x \triangleq \{ \bfx \in \cE_x \mid \lVert \bfx^0 - \bfx \rVert_x \le D_x \}$ and $\cB_y \triangleq \{ \bfy \in \cE_y \mid \lVert \bfy^0 - \bfy \rVert_y \le D_y \}$ denote the balls of radii $D_x, D_y$ around the initial points. Over the bounded domain $\cB \triangleq \cB_x \times \cB_y$, the duality gap is defined as
\[
\Delta( \bar \bfx, \bar \bfy ) \triangleq \max_{(\bfx, \bfy) \in \cB \cap Q} \bigl[ F(\bar \bfx, \bfy) - F(\bfx, \bar \bfy) \bigr] .
\]
We say that a pair $(\bar \bfx, \bar \bfy) \in Q$ is an $\epsilon$-\emph{saddle point} of Problem~\eqref{eq:saddle-point-problem} if $\Delta(\bar \bfx, \bar \bfy) \le \epsilon$. Our goal is to design an algorithm that produces such an $\epsilon$-saddle point for a given $\epsilon > 0$.

We remark that for the classic problem of constrained optimization with bounded domains, one can enclose the constrained sets in the balls $\cB_x$ and $\cB_y$ with sufficiently large radius~(e.g., the diameter of the constrained sets); the restricted saddle problem in form~\eqref{eq:saddle-point-problem} then coincides with the original one.

\paragraph{Problem class $\sP _{\textnormal{\upshape SP}}$.}

Assembling the function family $\sF$, the partial-gradient oracles $(\sO_x, \sO_y)$, the duality-gap accuracy measure $\Delta$, and a target accuracy $\epsilon > 0$, we obtain the problem class of interest, denoted by $\sP_{\text{\upshape SP}} (\sF, \sO_x, \sO_y, \Delta, \epsilon)$, or for short $\sP_{\text{\upshape SP}}$. Solving an instance $P \in \sP_{\text{\upshape SP}}$ requires an algorithm to output an $\epsilon$-saddle point of $\cF$ utilizing the distributed oracles.

To facilitate later discussion, we refer to the terms ${L_x D_x^2}$ and ${L_y D_y^2}$ as the \emph{diagonal conditioning}, and the term ${L_{xy} D_x D_y}$ as the \emph{cross-coupled conditioning}.

\subsection{Existing algorithms from literature}
\label{sec:existing-algorithms-for-SP}

In this section, we review existing algorithms for solving SPs and analyze their communication and oracle costs within the distributed method framework.
To keep the presentation concise, we summarize the methods and their limitations below, and defer their detailed algorithmic formulations, trajectories, and complexity propositions to \Cref{appendix:existing-algorithms}.

\textbf{Extragradient (\EG).} The classic \EG method~\citep{nemirovski2004prox,juditsky2011first} naturally fits our framework. Its distributed execution requires two communication rounds per iteration to evaluate coupled partial gradients at both the current and extrapolated points. It provides a robust and natural baseline for communication and oracle costs.

\textbf{Decoupled GDA (DGDA).} The DGDA method~\citep{zindari2025decoupled} attempts to reduce communication overhead by freezing the remote variable and taking multiple local gradient steps. While it achieves a fast logarithmic communication cost, its applications are highly restrictive: it only converges for weakly coupled strongly convex-strongly concave instances. For general problem class $\sP_{\text{\upshape SP}}$, the delayed remote variables cause the local updates to drift, leading the method to diverge.

\textbf{Catalyst acceleration.} Using a Catalyst wrapper around \EG (Cat-EG)~\citep{lin2020near,yang2020catalyst,lan2026novel} accelerates the algorithm's dependence on the diagonal conditioning. However, this comes with \textbf{five significant caveats}: (i)~it requires a complicated, multi-loop communication protocol and careful parameter tuning; (ii)~it is highly sensitive to the inexactness of the diameter estimates $\hat D_x$ and $\hat D_y$; (iii)~it introduces multiplicative logarithmic factors in the complexity; (iv)~under certain conditioning, its theoretical complexity can be strictly worse than the unaccelerated \EG baseline; and (v)~it does not support extensions to multi-agent scenarios~(cf. \cref{sec:monotone-block-VIPs}).

\textbf{Four-loop method.} The Cat-Cat-DAGDA method~\citep{wang2020improved} applies double Catalyst wrappers around a decoupled accelerated GDA to further accelerate the cross-coupling term. Despite this theoretical improvement, it shares all five caveats of Cat-EG, introduces even more complicated nested loops into the communication protocol, and adds further logarithmic factors. Consequently, it serves primarily as a theoretical benchmark rather than a practical method in our setting.

\textbf{Other distributed stochastic gradient methods.} Some recent papers~\citep{zhangcommunication2024,yoon2025multiplayer,yoon2025pearl} consider distributed SPs with stochastic gradient oracles. They propose different communication-efficient approaches; however, when applied to standard deterministic oracles considered in this paper, these methods fail to outperform \EG.

Consequently, as summarized in \cref{tab:communication-complexity}, the classic \EG method remains a formidable baseline for $\sP_{\text{\upshape SP}}$, and \emph{improving its communication and oracle complexity remains a significant challenge}.

\section{Decoupled method for SPs}
\label{sec:DM-for-SP}

When designing a communication-efficient method, the primary challenge is enabling distributed agents to compute local solutions independently despite the presence of cross-coupled functions. To address this, we propose a clean algorithmic template (or communication protocol) \emph{that reduces an SP into a sequence of coordinate-wise computational tasks}. We highlight the key results and insights below, deferring the detailed derivation to \cref{sec:multi-stage-reduction-for-distributed-SPs}.

\paragraph{Assembled norm.}

Given parameters $\alpha_x, \alpha_y > 0$ (to be specified later), we equip the joint space $\cE = \cE_x \times \cE_y$ with the assembled norm:
\begin{equation}\label{eq:assembled-norm-SP}
\Norm {\bfz} _{\cE} = {\InnerProduct {\bfP \bfz}{\bfz}} ^{\frac 1 2} = \sqrt { \alpha_x \Norm {\bfx} _x ^2 + \alpha_y \Norm {\bfy} _y^2 } \quad \text{for all $\bfz \in \cE$,}
\end{equation}
which corresponds to the block diagonal linear operator \( \bfP = \alpha_x \bfP_x \oplus \alpha_y \bfP_y \).

\paragraph{Template \DMsp.}

\Cref{alg:template-decoupled-reduced-operator-method-for-SP} outlines the Decoupled Method for Saddle Problems~(\DMsp), which adapts the abstract framework of the Reduced-Operator Method~\citep{nesterov2023high} for distributed environments. The algorithm maintains a sequence of anchor points $\bfv^t$ and proceeds iteratively.

First, the agents decouple the joint problem by fixing the remote variable at the current anchor $\bfv^t$. This allows Agent~$x$ and Agent~$y$ to independently and concurrently solve their respective regularized local subproblems up to target accuracies $\delta_x^{t+1}$ and $\delta_y^{t+1}$~(Lines~\ref{line:find-small-gradient-norm-SP-x} and~\ref{line:find-small-gradient-norm-SP-y}). Specifically, Agent~$x$ aims to approximately compute $\argmin _{\bfx \in \dom \psi_x} [ f (\bfx, \bfv^t_y) + \frac{\alpha_x \lambda_{t+1}}{2} \Norm{\bfx - \bfv^t_x}_x^2 + \psi_x (\bfx) ]$ by finding a point $\bfx^{t+1}$ whose regularized subgradient norm satisfies the exact mathematical bound specified in Line~\ref{line:find-small-gradient-norm-SP-x}. Agent~$y$ symmetrically performs an approximate minimization for its corresponding objective $-f (\bfv^t_x, \bfy) + \frac{\alpha_y \lambda_{t+1}}{2} \Norm{\bfy - \bfv^t_y}_y^2 + \psi_y (\bfy)$.

Following this local computation phase, the agents perform exactly two communication rounds to complete the iteration. In the first round (Line~\ref{line:exchange-variables}), the agents exchange their locally computed approximate solutions to assemble the joint intermediate point $\bfz^{t+1} = (\bfx^{t+1}, \bfy^{t+1})$. In the second round (Line~\ref{line:compute-extragradient}), they use this assembled point to evaluate their local partial gradients, which they then exchange to form the full joint operator $V_\psi(\bfz^{t+1})$.

Finally, using this assembled operator, the agents compute a closed-form step size $a_{t+1}$, update the running ergodic average $\bar{\bfz}^{t+1}$, and perform an extragradient-like step to generate the next anchor $\bfv^{t+1}$~(Lines~\ref{line:compute-reduced-stepsize} and~\ref{line:reduced-gradient-step-sp}). By structuring the method this way, \DMsp cleanly reduces the coupled SP into isolated coordinate-wise tasks with minimal communication overhead.

\begin{algorithm}[htb]
\caption{$\DMsp (f, (\psi_x, \psi_y), \bfz^0, (\lambda_{t})_{t \ge 1}, (\alpha_x, \alpha_y) )$}
\label{alg:template-decoupled-reduced-operator-method-for-SP}
\begin{algorithmic}[1]
    \State $\bfv ^{0} = (\bfv ^0 _x, \bfv ^0 _y) = \bfz ^0$.
    \For {$t=0,1,\dots,T-1$}
        \State Let $\delta^{t+1}_x = \frac {\alpha_x \lambda_{t+1}} {2}$ and $\delta^{t+1}_y = \frac {\alpha_y \lambda_{t+1}} {2}$.
        \State \label{line:find-small-gradient-norm-SP-x}Agent~$x$ finds $\bfx^{t+1}$ and $\psi_x^\prime(\bfx^{t+1}) \in \partial \psi_x (\bfx^{t+1})$ such that
        \[
            \Norm { \nabla _x f (\bfx^{t+1}, \bfv_y^t) + \alpha_x \lambda_{t+1}(\bfx^{t+1} - \bfv_x^t) + \psi_x^\prime(\bfx^{t+1}) } _{x^*} \le \delta^{t+1}_x \Norm {\bfx^{t+1} - \bfv^t_x} _x .
        \]
        \State \label{line:find-small-gradient-norm-SP-y}Agent~$y$ finds $\bfy^{t+1}$ and $\psi_y^\prime(\bfy^{t+1}) \in \partial \psi_y (\bfy^{t+1})$ such that
        \[
            \Norm { -\nabla _y f (\bfv^{t}_x, \bfy^{t+1}) + \alpha_y \lambda_{t+1}(\bfy^{t+1} - \bfv_y^t) + \psi_y^\prime(\bfy^{t+1}) } _{y^*} \le \delta^{t+1}_y \Norm {\bfy^{t+1} - \bfv^t_y} _y .
        \]
        \State \label{line:exchange-variables}Exchange $\bfx^{t+1}$ and $\bfy^{t+1}$ to assemble $\bfz^{t+1} = (\bfx^{t+1}, \bfy^{t+1})$.
        \State \label{line:compute-extragradient}Calculate corresponding coordinates of $V_\psi(\bfz^{t+1})$, then exchange to assemble:
        \[
            V _\psi (\bfz^{t+1}) = \bigl( \nabla _x f(\bfz^{t+1}) + \psi_x^\prime(\bfx^{t+1}), -\nabla_y f(\bfz^{t+1}) + \psi_y^\prime(\bfy^{t+1}) \bigr) .
        \]
        \State \label{line:compute-reduced-stepsize}Let $a_{t+1} = \frac { 2 \InnerProduct {V _\psi (\bfz^{t+1})}{ \bfv^t - \bfz^{t+1} } } { \Norm {V _\psi (\bfz^{t+1})} _{\cE^*} ^2 }$ and generate solution $\bar \bfz^{t+1} = \bigl( \sum _{i=1} ^{t+1} a_{i} \bigr) ^{-1} \sum _{i=1} ^{t+1} a_{i} \bfz^{i}$.
        \State \label{line:reduced-gradient-step-sp}$\bfv^{t+1} = \argmin _{\bfv \in Q}\ \bigl[ a_{t+1} \InnerProduct {V _\psi (\bfz^{t+1})}{\bfv} + \frac 1 2 \Norm {\bfv - \bfv^t} _{\cE} ^2 \bigr]$.
    \EndFor
\end{algorithmic}
\end{algorithm}

We refer to \cref{alg:template-decoupled-reduced-operator-method-for-SP} as a template method because we have not yet specified the implementations for the local computations in Lines~\ref{line:find-small-gradient-norm-SP-x} and~\ref{line:find-small-gradient-norm-SP-y}.
Provided that the inner local solvers in Lines~\ref{line:find-small-gradient-norm-SP-x} and~\ref{line:find-small-gradient-norm-SP-y} are standard gradient-based solvers, the template $\DMsp$ procedure formally qualifies as a distributed method.

\begin{theorem}
\label[theorem]{thm:DROMsp-communication-complexity}
    Consider the $\DMsp$ template applied to $\sP _{\text{\upshape SP}}$, assuming its local trajectories satisfy \Cref{def:distributed-gradient-based}.
    With the parameter choices of $\alpha_x = \frac {L_{xy} D_y} {D_x}$, $\alpha_y = \frac {L_{xy} D_x} {D_y}$, and $\lambda_t \equiv \lambda = 2$, we have:
    \[
    T ^{\DMsp} _{\sP_{\text{\upshape SP}}} \le 2 + 4 \frac{L_{xy} D_x D_y}{\epsilon} .
    \]
\end{theorem}

\begin{remark}[Communication improvement]
    \Cref{thm:DROMsp-communication-complexity} shows that the communication cost of \DMsp depends only on the cross-coupled conditioning $L_{xy} D_x D_y$, independent of the diagonal conditioning.
    In contrast, none of the existing methods reviewed in \cref{sec:existing-algorithms-for-SP} has achieved this sharp communication guarantee. Specifically, the communication cost of the \EG baseline is suboptimal due to its dependence on the diagonal conditioning. While advanced frameworks like Cat-Cat-DAGDA successfully isolate the communication cost from the diagonal conditioning, they suffer from highly complicated nested-loop designs and introduce poly-logarithmic overheads.
    Therefore, our \DMsp communication protocol represents a clear improvement over existing methods.
    Furthermore, as shown later in \cref{sec:lower-complexity-bounds-for-distributed-SP}, our $\cO \bigl( \frac {L_{xy} D_x D_y} \epsilon \bigr)$ communication cost is minimax optimal within the family of distributed gradient-span algorithms.
\end{remark}

\begin{remark}[Robustness to inexact distance estimates]\label{remark:robustness-to-inexact-distance-estimates}
    Let us consider a practical scenario where the algorithm may not have the precise values of $D_x$ and $D_y$ in advance, but it has access to upper estimates $\hat D_x \ge D_x$ and $\hat D_y \ge D_y$.
    Let
    \(
    \boxed{
        \theta \triangleq \frac{D_x \hat D_y} {\hat D_x D_y} + \frac{D_y \hat D_x} {\hat D_y D_x}
    },
    \)
    which quantifies the disproportionality between the true distance parameters and their estimates.
    Note that $\theta \ge 2$ with equality if and only if the estimates are proportional, i.e., $\hat D_x / D_x = \hat D_y / D_y$.
    Now, consider the $\DMsp$ template applied to $\sP_{\text{\upshape SP}}$.
    With the parameter choices of $\alpha_x = \frac {L_{xy} \hat D_y} {\hat D_x}$, $\alpha_y = \frac {L_{xy} \hat D_x} {\hat D_y}$, and $\lambda_t \equiv \lambda = 2$, we have:
    \[
    T ^{\DMsp} _{\sP_{\text{\upshape SP}}} \le 2 + 2 \theta \frac{L_{xy} D_x D_y}{\epsilon} .
    \]
    In particular, $\theta$ provides a \emph{scale-invariant} robustness compared to existing accelerated frameworks. As shown in \cref{tab:communication-complexity}, the communication complexities of Cat-EG and Cat-Cat-DAGDA scale directly with the product of the estimates, $\hat D_x \hat D_y$. Consequently, if both agents conservatively overestimate their domain sizes by a uniform factor $c \gg 1$ (i.e., $\hat D_x = c D_x$ and $\hat D_y = c D_y$), the communication cost of Catalyst-based methods inflates by a massive factor of $c^2$. For \DMsp, however, this uniform overestimation perfectly cancels out, yielding $\theta = 2$.
\end{remark}

\section{Novel multi-stage reduction and its building components}
\label{sec:multi-stage-reduction-for-distributed-SPs}

In this section, we reveal the technical components of our \DMsp, which is built upon a novel multi-stage reduction. We first leverage the Reduced-Operator Method to reduce the problem to a Monteiro-Svaiter Subproblem~(MSS). Then and crucially, we show that when this subproblem is weakly coupled, it can be solved by a Fully Decoupled Solver in one communication round. Consequently, the problem is further reduced to coordinate-wise Minimization of Residual Norms~(MRNs). Finally, by exploiting the strong maximal monotonicity, the agents can apply existing accelerated methods to solve the MRNs to desired accuracy.

\subsection{Preliminary: Variational inequality problems}
\label{sec:Preliminary-composite-VIPs}

Let us first introduce the general notion of composite variational inequality problem~(VIP) as the backbone of our problems.

For any operator $V ( \cdot ) \colon \dom \psi \to \cE^*$ and any function $\psi (\cdot) \colon \cE \to \R \cup \{ +\infty \}$, we say that $\bfz^* \in \cE$ is a (strong) \emph{solution} of the VIP of $( V, \psi )$ if
\begin{equation}
\label{eq:composite-VIP}
\InnerProduct {V ( \bfz^* )}{\bfz - \bfz^*} + \psi  ( \bfz ) \ge \psi  ( \bfz^* ) , \text{ for all } \bfz \in \dom \psi .
\end{equation}

\paragraph{Assumption for VIPs.}

Let us introduce the following assumption:
\begin{enumerate}[label={\textnormal{(A\arabic*')}}, nosep, leftmargin=*]
    \item \label{item:monotone-operator}
    The function $\psi$ is a (simple) proper closed convex function. The operator $V$ is continuous and monotone over $\dom \psi$: that is,
    \(
    \InnerProduct {V(\bfz^\prime) - V(\bfz) } {\bfz^\prime - \bfz} \ge 0, \text{ for all } \bfz^\prime, \bfz \in \dom \psi .
    \)
    \setcounter{VIP-assumptions}{\value{enumi}}
\end{enumerate}

There is another notion of a \emph{weak} solution with the alternative formulation:
\begin{equation*}
    \InnerProduct {V ( \bfz )}{\bfz - \bfz^*} + \psi  ( \bfz ) \ge \psi  ( \bfz^* ) , \text{ for all } \bfz \in \dom \psi .
\end{equation*}
Under Assumption~\ref{item:monotone-operator}, the weak and strong solutions are equivalent. We refer the reader to standard texts such as \cite{nemirovski2004prox,nesterov2023high} for a formal discussion of these two formulations.
Moreover, under \labelcref{item:monotone-operator}, a point  $\bfz^* \in \cE$ is the solution of \eqref{eq:composite-VIP} if and only if $\0 \in V (\bfz^*) + \partial \psi (\bfz^*)$.

Indeed, associated with the saddle problem given by $(f, \psi_x, \psi_y)$, let us consider the VIP given by $(V^f, \psi_z)$, where
\begin{equation}
\label{eq:VIP-arising-from-SP}
V^f(\bfz) \equiv (\nabla _x f (\bfz), - \nabla _y f (\bfz))
\text{ and }
\psi _z (\bfz) \equiv \psi_x (\bfx) + \psi_y (\bfy) , \text{ for all } \bfz \in Q .
\end{equation}
The associated VIP satisfies \labelcref{item:monotone-operator} whenever the saddle problem satisfies \labelcref{item:continuous-and-convex-concave}. Hence, the solution of the VIP coincides with the saddle point~\citep{nemirovski2004prox,nesterov2023high}.

\subsection{Reduced-operator method for VIPs}
\label{sec:reduced-operator-method}

Now, we introduce the Reduced-Operator Method~(\ROM) recently proposed in \cite{nesterov2023high}. In particular, we apply Nesterov's general framework in a special way so as to reduce the VIP to a sequence of {M}onteiro-{S}vaiter {S}ubproblems~(MSSs)~\citep{monteiro2013accelerated}.

The MSS asks to find a point for the regularized function such that the subgradient norm at this point is small compared to the distance from the initial point.
Let us now define the MSS formally.

\begin{mdframed}[skipabove=1mm]
\textbf{Monteiro-Svaiter Subproblem.}
Given an operator $V \colon \dom \psi \to \cE^*$, a function $\psi \colon \cE \to \R \cup \{ +\infty \}$, a reference point $\bfv \in \dom \psi$, and a real number $\lambda > 0$, we say $(\bfz^+, \psi^\prime (\bfz^+))$ is a \emph{solution} of the MSS if $\bfz^+ \in \dom \psi$, $\psi^\prime (\bfz^+) \in \partial \psi (\bfz^+)$, and
\begin{equation}
\label{eq:stopping-criterion-in-hybrid-projection}
    \lVert {V (\bfz^+) + \psi^\prime (\bfz^+) + \lambda \bfP (\bfz^+ - \bfv)} \rVert _{\cE^*} \le \lambda \lVert \bfz^+ - \bfv \rVert _{\cE} .
\end{equation}
\end{mdframed}

We will discuss how to solve the MSSs later in \cref{sec:ms-subproblem-with-weak-couplings}.
But for now, let us assume there exists a solver
\[
\cM^{\MS}(V, \psi, \bfv, \lambda)
\]
for the MSSs, which takes an MSS given by $(V, \psi, \bfv, \lambda)$ and returns a solution of it.
Built upon such a solver $\cM^{\MS}$, we now introduce \ROM in \cref{alg:Reduced-operator-method-for-VIPS}. At each iteration $t$: the solver $\cM ^{\MS}$ returns a solution $( \bfz^{t+1}, \psi^\prime (\bfz^{t+1}) )$ for the MSS built at reference point $\bfv^t$; this solution is used as a midpoint to compute subgradient $V_\psi (\bfz^{t+1})$; then the `extragradient-type' step is taken with the stepsize $a_{t+1}$.

\setlength{\textfloatsep}{0.2cm}
\begin{algorithm}
\caption{$\ROM _{\Norm{ \cdot } _{\cE} } ( V, \psi, \bfz^0, (\lambda_{t})_{t \ge 1}  \mid \cM^{\MS})$}
\label{alg:Reduced-operator-method-for-VIPS}
\begin{algorithmic}[1]
    \Require A solver $\cM^{\MS}$ for the MSSs.
    \State $\bfv^0 = \bfz^0$.

    \For { $t = 0, 1, \cdots$ }
        \State \label{line:Monteiro-Svaiter-subproblem}$( \bfz^{t+1}, \psi^\prime (\bfz^{t+1}) ) = \cM^{\MS} ( V, \psi, \bfv^t, \lambda_{t+1} )$.
        \State $V_\psi (\bfz^{t+1}) = V (\bfz^{t+1}) + \psi^\prime (\bfz^{t+1})$.
        \State \label{line:stepsize-a}$a_{t+1} = \frac { 2 \InnerProduct {V_\psi (\bfz^{t+1})}{ \bfv^t - \bfz^{t+1} } } { \lVert {V_\psi (\bfz^{t+1})} \rVert _{\cE^*} ^2 }$.
        \State \label{line:reduced-gradient-step}$\bfv^{t+1} = \argmin _{\bfv \in \dom \psi} \bigl[ a_{t+1} \InnerProduct {V _\psi (\bfz^{t+1})}{\bfv} + \frac 1 {2} \lVert \bfv - \bfv^t \rVert _{\cE} ^2 \bigr]$.
    \EndFor

\end{algorithmic}
\end{algorithm}
\setlength{\floatsep}{0.2cm}

Next, let us show the convergence of \ROM.
\begin{lemma}
\label[lemma]{lemma:Reduced-gradient-method-for-VIPs}
    \ROM~(\cref{alg:Reduced-operator-method-for-VIPS}) ensures for all $\bfv \in \dom \psi$ and for all $T \ge 1$,
    \[
    \sum _{t=0} ^{T-1} a_{t+1} \InnerProduct {V _\psi ( \bfz^{t+1} ) }{ \bfz^{t+1} - \bfv }
    \le \frac 1 2 \lVert \bfv^0 - \bfv \rVert _\cE ^2 - \frac 1 2 \lVert \bfv^T - \bfv \rVert _\cE ^2 .
    \]
    Moreover, we have $a_{t+1} \ge \frac 1 {\lambda_{t+1}}$, for all $t \ge 0$.
\end{lemma}
\begin{proof}
    By the optimality of $\bfv ^{t+1}$, we have for all $\bfv \in \dom \psi$,
    \[
    a_{t+1} \InnerProduct {V_\psi ( \bfz^{t+1} ) }{ \bfv - \bfv^{t+1} } + \frac 1 2 \lVert \bfv^t - \bfv \rVert _\cE ^2
    \ge \frac 1 2 \lVert {\bfv^{t+1} - \bfv} \rVert _\cE ^2 + \frac 1 2 \lVert {\bfv^t - \bfv^{t+1}} \rVert _\cE ^2 ,
    \]
    and therefore,
    \[
    \begin{aligned}
        &a_{t+1} \InnerProduct {V _\psi ( \bfz^{t+1} )}{\bfv - \bfz^{t+1} } + \frac 1 2 \lVert \bfv^t - \bfv \rVert _\cE ^2 \\
        &\quad \ge a_{t+1} \InnerProduct {V _\psi (\bfz^{t+1})}{ \bfv^{t+1} - \bfz^{t+1} } + \frac 1 2 \Norm {\bfv^{t+1} - \bfv} _\cE ^2 + \frac 1 2 \Norm {\bfv^t - \bfv^{t+1}} _\cE ^2 \\
        &\quad = a_{t+1} \InnerProduct {V _\psi (\bfz^{t+1})}{ \bfv^{t} - \bfz^{t+1} } + \frac 1 2 \Norm {\bfv^{t+1} - \bfv} _\cE ^2 + a_{t+1} \InnerProduct {V _\psi (\bfz^{t+1})}{ \bfv^{t+1} - \bfv^{t} } + \frac 1 2 \Norm {\bfv^t - \bfv^{t+1}} _\cE ^2 \\
        &\quad \ge a_{t+1} \InnerProduct {V _\psi (\bfz^{t+1})}{ \bfv^{t} - \bfz^{t+1} } + \frac 1 2 \Norm {\bfv^{t+1} - \bfv} _\cE ^2 - \frac {a_{t+1}^2} {2} \Norm { V_\psi (\bfz^{t+1}) } _{\cE^*} ^2 \\
        &\quad \ge \frac 1 2 \Norm {\bfv^{t+1} - \bfv} _\cE ^2 ,
    \end{aligned}
    \]
    where the last inequality follows from the definition of $a_{t+1}$ in Line~\ref{line:stepsize-a} of \cref{alg:Reduced-operator-method-for-VIPS}.
    Then, the desired bound follows from summing the above inequality over $t$ from $0$ to $T-1$.

    Next, we show the lower bound for $a_{t}$.
    For all $t \ge 1$,
    we have
    \[
    \begin{aligned}
    &\quad \InnerProduct {V_\psi (\bfz^{t})}{ \bfv^{t-1} - \bfz^{t} } - \frac 1 {2 \lambda_{t}} \Norm {V_\psi (\bfz^{t})} _{\cE^*} ^2 \\
    &\equiv \frac {\lambda_{t}} {2} \Norm {\bfz^{t} - \bfv^{t-1}} _{\cE} ^2 - \frac {1} {2 \lambda_{t}} \Norm { V _\psi (\bfz^{t}) + \lambda_{t} \bfP (\bfz^{t} - \bfv^{t-1})} _{\cE^*} ^2 \\
    &\ge 0 ,
    \end{aligned}
    \]
    where the last inequality follows from \cref{eq:stopping-criterion-in-hybrid-projection}.
    Therefore, we have
    \[
    a_{t} = \frac { 2 \InnerProduct {V_\psi (\bfz^{t})}{ \bfv^t - \bfz^{t} } } { \Norm {V_\psi (\bfz^{t})} _{\cE^*} ^2 } \ge \frac {1} {\lambda_{t}} .
    \]
\end{proof}

\subsection{Fully decoupled solver for MSSs with weak couplings}
\label{sec:ms-subproblem-with-weak-couplings}

We now address the MSS introduced by the ROM in \cref{sec:reduced-operator-method}. Specifically, we are to deal with the MSS given by
\begin{equation}
\label{eq:MS-subproblem-from-SPs}
(V^f, \psi_z, \bfv, \lambda) ,
\end{equation}
where
$V^f$ and $\psi_z$ are defined in \cref{eq:VIP-arising-from-SP},
the reference point $\bfv = (\bfv_x, \bfv_y) \in \dom \psi_x \times \dom \psi_y$,
and the assembled norm $\Norm {\cdot} _{\cE}$ is associated with parameters $\alpha_x$ and $\alpha_y$.

We say that an MSS has a \emph{weak coupling} if $\lambda \ge \boxed {2 \bar L_{\texttt{c}} \triangleq \frac{2 L_{xy}} {\sqrt{\alpha_x \alpha_y}}}.$
In this section, we will introduce a {F}ully {D}ecoupled {S}olver~(\FDS), which reduces the weakly-coupled MSSs to coordinate-wise Minimization of Residual Norms~(MRNs).

Let us first define the problem of MRN.
The MRN asks to find an approximate solution $\bfw^+$ of the VIP of $(V_w, \psi_w)$ such that the residual norm is small compared to the distance from the initial point:
\begin{mdframed}[skipabove=1mm]
\textbf{Minimization of residual norm.}
Given an operator $V _w \colon \dom \psi _w \to
\cE_w^*$, a function $\psi _w \colon \cE_w \to \R \cup \{ +\infty \}$, a reference point $\bfv_w \in \dom \psi _w$, and an accuracy $\delta > 0$, we say $(\bfw^+, \psi _w ^\prime (\bfw^+))$ minimizes the residual norm to \emph{$\delta$-relative distance accuracy}, if $\bfw^+ \in \dom \psi _w$, $\psi _w ^\prime (\bfw^+) \in \partial \psi _w (\bfw^+)$, and
\[
\lVert V _w (\bfw^+) + \psi _w ^\prime (\bfw^+) \rVert _{w^*} \le \delta \lVert  \bfw^+ - \bfv_{w} \rVert _{w} .
\]
\end{mdframed}

Let us, again, defer the discussion of solving MRNs to \cref{sec:minimizing-residual-norm}. But for now, let us assume there exist solvers
\[
\cM_x ^{\MRN}(V_x, \hat \psi_x, \bfv_x, \delta_x) \quad \text{and} \quad \cM_y ^{\MRN}(V_y, \hat \psi_y, \bfv_y, \delta_y)
\]
for the coordinate-wise MRNs in spaces $\cE_x$ and $\cE_y$. In particular, these solvers take a coordinate-wise MRN problem and return a point and a subgradient satisfying the desired accuracy.

A crucial step in our analysis relies on the Fully Decoupled Solver (\FDS, \cref{alg:fully-decoupled-solver}), which optimizes each decision variable independently. For an MSS with weak coupling, we establish in \cref{lemma:FDS-correctness} that \FDS returns a correct solution \emph{in a single round}.

\begin{algorithm}[htb]
\caption{$\FDS _{\Norm {\cdot}_{\cE}} ( (\nabla _x f, - \nabla _y f), (\psi_x, \psi_y), \bfv, \lambda \mid (\cM_x ^{\MRN}, \cM_y ^{\MRN}) )$}
\label{alg:fully-decoupled-solver}
\begin{algorithmic}[1]
\Require Solvers $\cM_x ^{\MRN}$ and $\cM_y ^{\MRN}$ for the coordinate-wise MRNs.
        \State $\hat \psi_x = \psi_x + \frac{\alpha_x \lambda}{2} \Norm {\cdot - \bfv_x} _x ^2$ and $\hat \psi_y = \psi_y + \frac{\alpha_y \lambda}{2} \Norm {\cdot - \bfv_y} _y ^2$.
        \State \label{line:coordinate-MRNs-in-FDS}Agent~$x$ and Agent~$y$ respectively compute
            \[
            \begin{gathered}
            (\bfx^+, \hat \psi^\prime _x (\bfx^+)) = \cM_x ^{\MRN} \bigl(
                \nabla _x f (\cdot, \bfv_y),
                \hat \psi_x,
                \bfv_x,
                \delta_x \bigr) \text{ and } \\
            (\bfy^+, \hat \psi^\prime _y (\bfy^+)) = \cM_y ^{\MRN} \bigl(
                - \nabla _y f (\bfv_x, \cdot),
                \hat \psi_y,
                \bfv_y,
                \delta_y \bigr) ,
            \end{gathered}
            \]
            where $\delta_x = \frac {\alpha_x \lambda} {2}$ and $\delta_y = \frac {\alpha_y \lambda} {2}$.
    \State $\psi^\prime_x (\bfx^+) = \hat \psi^\prime_x (\bfx^+) - \alpha_x \lambda \bfP_x (\bfx^+ - \bfv_x)$ and $\psi^\prime_y (\bfy^+) = \hat \psi^\prime_y (\bfy^+) - \alpha_y \lambda \bfP_y (\bfy^+ - \bfv_y)$.
    \State \Return $(\bfz^+, \psi ^\prime (\bfz^+) )$, where $ \bfz^+ = (\bfx^+, \bfy^+)$ and $\psi ^\prime (\bfz^+) = (\psi ^\prime _x (\bfx^+), \psi ^\prime _y (\bfy^+))$.
\end{algorithmic}
\end{algorithm}

\begin{lemma}
\label[lemma]{lemma:FDS-correctness}
    Consider the saddle problem given by $(f, \psi_x, \psi_y)$ which satisfies \labelcref{item:function-smooth}.
    For $\lambda \ge 2 \bar L_{\texttt{c}}$, \FDS~(\cref{alg:fully-decoupled-solver}) returns a solution of the MSS given in \cref{eq:MS-subproblem-from-SPs}.
\end{lemma}

The correctness of the \FDS for SPs can be implied as a direct consequence of the correctness of a more general version of \FDS for VIPs, which will be introduced later in \cref{sec:monotone-block-VIPs-appendix}. Therefore, let us defer this proof to \cref{lemma:FDS-correctness-VIP} in \cref{sec:monotone-block-VIPs-appendix}.

\subsection{Minimization of residual norms}
\label{sec:minimizing-residual-norm}

We arrive at the last building component, the Minimization of Residual Norms~(MRNs).

\paragraph{Assumptions for MRNs.}

Let us introduce the following assumptions:
\begin{enumerate}[label={\textnormal{(\^{A}\arabic*)}}, nosep, leftmargin=*]
    \item \label{item:a-monotone-operator-and-a-convex-function} The function $\psi _w$ is a (simple) proper closed convex function. The operator $V _w$ is monotone over $\dom \psi _w$.

    \item \label{item:strongly-maximally-monotone} 
    The set-valued operator $V _w + \partial \psi _w$ is $\mu$-strongly maximally monotone over $\dom \psi _w$.
    That is,
    \[
    \InnerProduct [\big] {V _w (\bfw^\prime) + \psi _w ^\prime (\bfw^\prime) - V _w (\bfw) - \psi _w ^\prime (\bfw)} {\bfw^\prime - \bfw} \ge \mu \Norm {\bfw^\prime - \bfw} _w ^2,
    \]
    for all $\bfw^\prime, \bfw \in \dom \psi _w$, $\psi^\prime _w (\bfw^\prime) \in \partial \psi _w (\bfw^\prime)$, and $\psi^\prime _w (\bfw) \in \partial \psi _w (\bfw)$,

    \item \label{item:operator-Lipschitz-continuity} The operator $V _w (\bfw)$ is $L$-Lipschitz continuous over $\bfw \in \dom \psi _w$.

    \item \label{item:operator-is-gradient-of-convex-function} The operator $V _w = \nabla f_w$, where $f_w$ is a continuously differentiable function defined on an open set containing $\dom \psi_w$.
\end{enumerate}

The theoretical guarantee provided in the literature is usually based on the distance-to-solution accuracy~(cf. \cref{def:distance-to-solution-accuracy}). We show in \cref{lemma:exact-accuracy-to-relative-accuracy-under-strong-monotonicity} that, under strong maximal monotonicity, the relative distance accuracy required in this paper can be implied from distance-to-solution accuracy.
\begin{definition}[Distance-to-solution accuracy]
\label[definition]{def:distance-to-solution-accuracy}
    We say that $(\bfw^+, \psi _w ^\prime (\bfw^+))$ satisfies \emph{$\xi$-distance-to-solution accuracy} if $\bfw^+ \in \dom \psi _w$, $\psi _w ^\prime (\bfw^+) \in \partial \psi _w (\bfw^+)$, and \( \lVert {V _w (\bfw^+) + \psi _w ^\prime (\bfw^+) } \rVert _{w^*}
    \le \xi \lVert {\bfv_{w} - \tilde \bfw} \rVert _{w} \) for some $\tilde \bfw$ in the solution set of the VIP of $(V _w, \psi _w)$.
\end{definition}

\begin{lemma}
\label[lemma]{lemma:exact-accuracy-to-relative-accuracy-under-strong-monotonicity}
    Consider the MRN problem given by $(V_w, \psi_w, \bfv_w, \delta)$ which satisfies \labelcref{item:strongly-maximally-monotone}.
    Let $\xi \le \frac{\mu\delta} {\mu+\delta}$.
    If $(\bfw^+, \psi _w ^\prime (\bfw^+))$ satisfies {$\xi$-distance-to-solution accuracy}, then $(\bfw^+, \psi _w ^\prime (\bfw^+))$ is a solution of the MRN problem.
\end{lemma}

\begin{proof}
    In view of the triangle inequality and the $\mu$-strong maximal monotonicity of $V _w + \psi _w$, we have
    \[
    \begin{aligned}
    &\Norm {\bfv_{w} - \tilde \bfw} _{w}
    \le \Norm {\bfw^+ - \bfv_{w}} _{w} + \Norm {\bfw^+ - \tilde \bfw} _{w} \\
    &\quad \le \Norm {\bfw^+ - \bfv_{w}} _{w} + \frac 1 {\mu} \Norm { V _w (\bfw^+) + \psi _w ^\prime (\bfw^+) } _{w^*} \\
    &\quad \le \Norm {\bfw^+ - \bfv_{w}} _{w} + \frac {\xi } {\mu } \Norm {\bfv_{w} - \tilde \bfw} _{w} .
    \end{aligned}
    \]
    Then, we have
    \[
        \Norm {\bfv_{w} - \tilde \bfw} _{w} \le \frac {\mu } {\mu - \xi } \Norm {\bfw^+ - \bfv_{w}} _{w} .
    \]
    Therefore, we have
    \[
    \Norm { V _w (\bfw^+) + \psi _w ^\prime (\bfw^+) } _{w^*} \le \xi \Norm {\bfv_{w} - \tilde \bfw} _{w}
    \le \frac { \mu \xi } {\mu - \xi } \Norm {\bfw^+ - \bfv_{w}} _{w} \le \delta ,
    \]
    where the last inequality follows from the assignment $\xi \le \frac{\mu\delta} {\mu+\delta}$.
\end{proof}

We will leverage efficient existing solvers for the MRN problems. In particular, we are to deal with the specific MRNs given in Line~\ref{line:coordinate-MRNs-in-FDS} in \cref{alg:fully-decoupled-solver}, where the corresponding coordinate-wise operators are gradients of smooth convex functions.
Therefore, we can leverage the existing accelerated gradient methods from the literature.

Let us apply, for instance, the {A}ccumulative {R}egularization {M}ethod~(\ARM) from \cite{lan2023optimal}, whose detailed pseudocode is presented in \cref{alg:accumulative_regularization} in \cref{sec:MRN-solver-ARM} for completeness.
Let us denote this algorithm as
\begin{equation}
\label{eq:ARM-equation}
\ARM\ ( \nabla f _w, \psi _w, \bfv_w, \xi \mid L ),
\end{equation}
which takes an MRN instance of interest, has knowledge of the parameter $L$ in \labelcref{item:operator-Lipschitz-continuity}, and returns a solution that satisfies $\xi$-distance-to-solution accuracy.

Now, we state the oracle complexity of MRN with respect to the distance-to-solution accuracy. The original result of \cite{lan2023optimal} is given in projected gradient norm, which can be converted to the subgradient norm considered in this paper. We defer the detailed proof to \cref{sec:MRN-solver-ARM}.

\begin{lemma}[\citealt{lan2023optimal}]
\label[lemma]{lemma:making-the-gradient-norm-small-full}
    Assume \labelcref{item:a-monotone-operator-and-a-convex-function}, \labelcref{item:operator-Lipschitz-continuity}, \labelcref{item:operator-is-gradient-of-convex-function}, and that the solution set of the VIP of $(V_w, \psi _w)$ is non-empty.
    Let
    \[
    ( \bfw^+, \psi _w ^\prime (\bfw^+) ) = \ARM\ ( \nabla f _w, \psi _w, \bfv_w, \xi \mid L ).
    \]
    Then, \ARM takes no more than
    \(
    34 \cdot \sqrt { \frac {3 L} {2 \xi} } 
    \)
    queries to $\nabla f _w (\cdot)$ and ensures that $( \bfw^+, \psi _w ^\prime (\bfw^+) )$ satisfies $\xi$-distance-to-solution accuracy.
\end{lemma}
%


\section{Decoupled method for SPs: Concrete implementation}
\label{sec:implementable-algorithm-SP}

We are now back to considering the original SPs in \cref{eq:saddle-point-problem}. Let us combine the technical components in \cref{sec:multi-stage-reduction-for-distributed-SPs} and present the final, implementable version of \DMsp.

\paragraph{Implementation of \DMsp.}

We use the $\ARM$ solver in \cref{eq:ARM-equation} for Minimization of Residual Norms:
\begin{equation}\label{eq:MRN-solvers-SP}
\begin{gathered}
\cM ^{\MRN} _{x} ( V_x, \hat \psi_x, \bfv_x,\delta_x ) \triangleq \ARM \bigl( V_x, \hat \psi_x, \bfv_x, \frac {2 \delta_x} {3} \mid L_x \bigr), \\
\cM ^{\MRN} _{y} ( V_y, \hat \psi_y, \bfv_y, \delta_y ) \triangleq \ARM \bigl( V_y, \hat \psi_y, \bfv_y, \frac {2 \delta_y} {3} \mid L_y \bigr) .
\end{gathered}
\end{equation}
Consider the assembled norm $\Norm {\cdot} _{\cE}$ with parameters $\alpha_x$ and $\alpha_y$.
Then, for any Monteiro-Svaiter Subproblem given by $(V^f, \psi_z, \bfv, \lambda)$, we leverage the solver
\[
\algname{FDS-ARM} (V^f, \psi_z, \bfv, \lambda) = \FDS _{\Norm{\cdot}_{\cE}} \bigl(
    V^f, \psi_z, \bfv, \lambda
    \mid
    (\cM ^{\MRN} _{x}, \cM ^{\MRN} _{y} )
\bigr)
\]
Finally, we obtain the concrete algorithm \DMsp as follows:
\begin{mdframed}
\begin{equation}
\label{eq:final-implementation-of-DROMsp}
\ROM _{\Norm{\cdot}_{\cE}} \bigl(
    V^f, \psi_z, \bfz^0, (\lambda_{t})_{t \ge 1}
    \mid
    \algname{FDS-ARM}
\bigr) .
\end{equation}
\end{mdframed}
%

Combining \cref{lemma:Reduced-gradient-method-for-VIPs,lemma:FDS-correctness,lemma:exact-accuracy-to-relative-accuracy-under-strong-monotonicity,lemma:making-the-gradient-norm-small-full}, we are ready to prove the main convergence lemma for SPs in \cref{lemma:DROMsp-communication-complexity}.

\begin{lemma}
\label[lemma]{lemma:DROMsp-communication-complexity}
    Consider \DMsp with the implementation in \cref{eq:final-implementation-of-DROMsp}, applied to $\sP _{\text{\upshape SP}}$.
    Under \labelcref{item:bounded-distance-to-saddle-point,item:continuous-and-convex-concave,item:function-smooth}, for $\lambda_{t+1} \equiv \lambda \ge \frac {2 L_{xy} } { \sqrt {\alpha_x \alpha_y }} $, the algorithm takes no more than
    \(
        2 T
    \)
    communication rounds,
    no more than
    \[
       T \cdot \bigl( 1 + 34 \sqrt { \frac{9 L_x} {2 \alpha_x \lambda} } \bigr)
    \]
    queries to $\nabla _x f$, and no more than
    \[
       T \cdot \bigl( 1 + 34 \sqrt { \frac{9 L_y} {2 \alpha_y \lambda} } \bigr)
    \]
    queries to $- \nabla _y f$,
    and obtains an $\epsilon$-saddle point
    \(
    \bar \bfz^T ,
    \)
    where
    \[
    T = \Ceil[\Big]{\frac { \alpha_x \lambda D_x ^2 + \alpha_y \lambda D_y ^2 } {2 \epsilon}}.
    \]
\end{lemma}

\begin{proof}
    By \labelcref{item:continuous-and-convex-concave}, we have
    \[
    \Delta (\bar \bfz^{T}) \le \Bigl( \sum _{t=0} ^{T-1} a_{t+1} \Bigr) ^{-1} \max _{\bfz \in \cB \cap Q}\ \biggl[ \sum _{t=0} ^{T-1} a_{t+1} \InnerProduct { V _\psi (\bfz^{t+1}) }{ \bfz^{t+1} - \bfz } \biggr] .
    \]
    Further, with $\lambda \ge 2 \bar L_{\texttt{c}}$, by \cref{lemma:Reduced-gradient-method-for-VIPs,lemma:FDS-correctness}, we have
    \[
    \begin{aligned}
    & \Delta (\bar \bfz^{T})
    \le \Bigl( \sum _{t=0} ^{T-1} a_{t+1} \Bigr) ^{-1} \max _{\bfz \in \cB \cap Q}\ \biggl[ \sum _{t=0} ^{T-1} a_{t+1} \InnerProduct { V _\psi (\bfz^{t+1}) }{ \bfz^{t+1} - \bfz } \biggr] \\
    &\quad \le \Bigl( \sum _{t=0} ^{T-1} a_{t+1} \Bigr) ^{-1} \max _{\bfz \in \cB \cap Q}\ \Bigl[ { \frac {\alpha_x} {2} \Norm {\bfx^0 - \bfx}_x ^2 + \frac {\alpha_y} {2} \Norm {\bfy^0 - \bfy}_y ^2 } \Bigr] \\
    &\quad \le \Bigl( \sum _{t=0} ^{T-1} \frac 1 {\lambda_{t+1}} \Bigr) ^{-1} \cdot \frac 1 2 (\alpha_x D_x^2 + \alpha_y D_y^2)
    \le \epsilon ,
    \end{aligned}
    \]
    where the last inequality follows from the assignments of $(\lambda_{t}) _{t \ge 1}$ and $T$.
    Therefore, the number of communication rounds is bounded by $2T$.

    Now we count the number of gradient queries.
    By \cref{lemma:exact-accuracy-to-relative-accuracy-under-strong-monotonicity}, \ARM always returns the solution with the required relative distance accuracy; and in view of \cref{lemma:making-the-gradient-norm-small-full}, it takes no more than $34 \sqrt { \frac{9 L_x } {2 \alpha_x \lambda} }$ gradient queries to $\nabla _x f$ and no more than $34 \sqrt { \frac{9 L_y } {2 \alpha_y \lambda} }$ gradient queries to $- \nabla _y f$.
    Therefore, the numbers of gradient queries to $\nabla _x f$ and $- \nabla _y f$ are bounded by $T \cdot \bigl( 1 + 34 \sqrt { \frac{9 L_x } {2 \alpha_x \lambda} } \bigr)$ and $T \cdot \bigl( 1 + 34 \sqrt { \frac{9 L_y } {2 \alpha_y \lambda} } \bigr)$, respectively.
\end{proof}

Finally, we conclude with the following guarantee on the oracle cost.

\begin{theorem}
\label{thm:DROMsp-computational-complexity}
    Consider \DMsp with the implementation in \cref{eq:final-implementation-of-DROMsp}, applied to $\sP _{\text{\upshape SP}}$.
    With the same choices of $\alpha_x$, $\alpha_y$, and $(\lambda _t) _{t \ge 0}$ as in \cref{thm:DROMsp-communication-complexity},
    we have:
    \[
    N _{\sP_{\text{\upshape SP}}} ^{\DMsp} = (c_x + c_y) \frac{2 L_{xy} D_x D_y}{\epsilon}
    + 102 \Bigl( \frac{L_{xy} D_x D_y}{\epsilon} \Bigr) ^{\frac 1 2} \Bigl( c_x \Bigl( \frac{L_x D_x^2}{\epsilon} \Bigr) ^{\frac 1 2} + c_y \Bigl( \frac{L_y D_y^2}{\epsilon} \Bigr) ^{\frac 1 2} \Bigr) .
    \]
\end{theorem}

As detailed in \cref{sec:existing-algorithms-for-SP}, alternative distributed algorithms either fail to outperform \EG in general for $\sP _{\text{\upshape SP}}$ or suffer from certain theoretical and practical caveats. Consequently, the \EG method remains a crucial baseline for oracle complexity, which we now compare against.

\begin{remark}[Oracle improvement]
    Let us compare the oracle cost of \DMsp against the \EG baseline, which requires
    \[
    N^{\EG} _{{\sP}^\circ _{\text{\upshape SP}}} = (c_x + c_y) \cdot \Bigl( \frac{L_{xy} D_x D_y} {\epsilon} + \frac {L_x D_x^2} {\epsilon} + \frac {L_y D_y^2} {\epsilon} \Bigr)
    \]
    oracle costs~(\cref{proposition:EG-upper-bound-SP}). We conclude that the computational cost of \DMsp is consistently bounded by that of \EG. Furthermore, it yields a substantial improvement when
    \[
    L_x D_x^2 + L_y D_y^2 \gg L_{xy} D_x D_y + \sqrt {L_{xy} D_x D_y} \cdot \Bigl( \frac { c_x } {c_x + c_y} \sqrt {L_x D_x^2 } + \frac { c_y } {c_x + c_y} \sqrt {L_y D_y^2 } \Bigr) .
    \]
    For instance, assuming uniform oracle costs ($c_x = c_y$), this improvement occurs when the diagonal conditioning dominates the cross-coupled conditioning, i.e., $L_x D_x^2 + L_y D_y^2 \gg \sqrt {L_{xy} D_x D_y}$. To our knowledge, \DMsp is the first method to consistently improve upon the \EG oracle cost for $\sP_{\text{\upshape SP}}$.
\end{remark}

\section{Lower complexity bounds for SPs}
\label{sec:lower-complexity-bounds-for-distributed-SP}

In this section, we establish lower complexity bounds for distributed SPs. In particular, we focus on algorithms in the distributed gradient-span framework~\citep{nesterov2004introductory}. This restriction may not be absolutely necessary, and we might be able to avoid it via more complicated reasoning. However, it naturally holds for the majority of practical algorithms as detailed in \Cref{appendix:existing-algorithms}. Hence, we consider this as a crucial first step towards establishing information-theoretic lower bounds for distributed SPs.

To establish the lower bounds, we analyze a subclass of problems, denoted by $\sP_{\text{\upshape SP}}^\circ \subset \sP_{\text{\upshape SP}}$, where the local regularizing components are identically zero (i.e., $\psi_x \equiv 0$ and $\psi_y \equiv 0$). Establishing lower bounds on this unconstrained, smooth subclass yields lower bounds for the broader class $\sP_{\text{\upshape SP}}$. Because the local functions are zero, the agents' updates rely exclusively on the partial gradients of the coupling function $f$, and no subgradients are involved.

For the lower bound analysis, we consider the case where the messages exchanged between rounds allow each agent to reconstruct the other's historical gradients (which can be seen as the worst case for lower bound analysis). Thus, Agent~$x$ utilizes its own gradients up to the current local step, and Agent~$y$'s gradients up to the end of the previous round.

\subsection{Distributed gradient-span framework}
Building upon the information-based complexity established in \Cref{sec:distributed-gradient-based-algorithms}, the gradient-span assumption imposes a structural restriction on the deterministic mappings. We assume that the iterates are constructed by taking linear combinations of the historically observed preconditioned partial gradients.

\paragraph{Algorithm trajectories and histories.}
To describe the available gradients at any given step, we use a more specific notation for the local query trajectories. In each round $t \in \{0, \dots, T-1\}$, Agent~$x$ and Agent~$y$ generate local query trajectories of lengths $\tau_x^t$ and $\tau_y^t$. The query points consist of the agent's local variable and its delayed approximation of the remote variable:
\[
\hat Z_x^t = \bigl\{ \bfz_x^{t,l} = (\bfx^{t, l}, \hat \bfy^{t, l}) \bigr\} _{ l = 0 }^{\tau_x^t - 1}
\quad \text{and} \quad
\hat Z_y^t = \bigl\{ \bfz_y^{t,l} = (\hat \bfx^{t, l}, \bfy^{t, l}) \bigr\} _{l = 0 }^{\tau_y^t - 1} .
\]
Let $Z_x^{t-1} = \bigcup_{i=0}^{t-1} \hat Z_x^i$ and $Z_y^{t-1} = \bigcup_{i=0}^{t-1} \hat Z_y^i$ denote the accumulated sets of query points from all prior rounds (with $Z_x^{-1} = Z_y^{-1} = \emptyset$). Within round $t$, the sets of queries made up to local step $l$ are denoted by $Z_x^{t,l} = Z_x^{t-1} \cup \{ \bfz_x^{t, i} \}_{i=0}^{l-1}$ and $Z_y^{t,l} = Z_y^{t-1} \cup \{ \bfz_y^{t, i} \}_{i=0}^{l-1}$. Finally, the accumulated queries up to the end of round $t$ are $Z_x^t = Z_x^{t, \tau_x^t}$ and $Z_y^t = Z_y^{t, \tau_y^t}$.

\begin{assumption}[Local variables]
\label{assum:local-variables}
For any instance $P \in \sP_{\text{\upshape SP}}^\circ$ and all rounds $t \in \{0, \dots, T-1\}$, the local variables queried by Agent~$x$ and Agent~$y$ satisfy, for all respective local steps $l$:
\begin{align*}
    \bfx ^{t,l} &\in \bfx^0 + \bfP_x^{-1} \spn \bigl\{ \nabla _x f ( \bfz ) \mid \bfz \in Z_x^{t,l} \bigr\}, \\
    \bfy ^{t,l} &\in \bfy^0 + \bfP_y^{-1} \spn \bigl\{ \nabla _y f ( \bfz ) \mid \bfz \in Z_y^{t,l} \bigr\} .
\end{align*}
\end{assumption}

\begin{assumption}[Remote variables]
\label{assum:remote-variables}
For any instance $P \in \sP_{\text{\upshape SP}}^\circ$ and all rounds $t \in \{0, \dots, T-1\}$, the delayed approximations of the remote variables queried by Agent~$x$ and Agent~$y$ satisfy, for all respective local steps $l$:
\begin{align*}
    \hat \bfy ^{t,l} &\in \bfy^0 + \bfP_y^{-1} \spn \bigl\{ \nabla _y f ( \bfz ) \mid \bfz \in Z_y^{t-1} \bigr\} , \\
    \hat \bfx ^{t,l} &\in \bfx^0 + \bfP_x^{-1} \spn \bigl\{ \nabla _x f ( \bfz ) \mid \bfz \in Z_x^{t-1} \bigr\} .
\end{align*}
\end{assumption}

\begin{assumption}[Candidate solutions]
\label{assum:generated-solutions}
For any instance $P \in \sP_{\text{\upshape SP}}^\circ$ and all rounds $t \in \{0, \dots, T-1\}$, the candidate solutions $\bar \bfz^{t+1} = (\bar \bfx^{t+1}, \bar \bfy^{t+1})$ generated at the end of round $t$ satisfy
\begin{align*}
    \bar \bfx^{t+1} &\in \bfx^0 + \bfP_x^{-1} \spn \bigl\{ \nabla _x f ( \bfz ) \mid \bfz \in Z_x^t \bigr\} , \\
    \bar \bfy^{t+1} &\in \bfy^0 + \bfP_y^{-1} \spn \bigl\{ \nabla _y f ( \bfz ) \mid \bfz \in Z_y^t \bigr\} .
\end{align*}
\end{assumption}

\begin{definition}
\label{def:gradient-span}
An algorithm is called \emph{gradient-span} if \cref{assum:local-variables,assum:remote-variables,assum:generated-solutions} are all satisfied.
\end{definition}

These assumptions specify that updates are confined to the available gradient subspaces. The impact of delayed communication is modeled by restricting the cross-variable approximations ($\hat{\bfy}$ and $\hat{\bfx}$) to the remote gradients from $Z_y^{t-1}$ and $Z_x^{t-1}$, respectively.

This gradient-span family includes the majority of practical algorithms, such as \EG, Decoupled-GDA, Cat-EG, and Cat-Cat-DAGDA.
A more detailed discussion is provided in \cref{appendix:existing-algorithms}.

Finally, we remark that our proposed algorithm, \DMsp, also operates within this framework. Specifically, each iteration $t$ of \DMsp requires exactly two communication rounds. In the first round, the local solvers perform multiple steps to compute the approximate solutions $\bfx^{t+1}$ and $\bfy^{t+1}$, generating local trajectories denoted by $\hat Z_x^{2t}$ and $\hat Z_y^{2t}$. In the second round, the agents evaluate partial gradients at the assembled point $\bfz^{t+1}$, yielding singleton trajectories $\hat Z_x^{2t+1} = \hat Z_y^{2t+1} = \{ \bfz^{t+1} \}$.

\subsection{Lower complexity bounds}
\label{sec:proof-for-lower-complexity-bounds}

Now we provide detailed proofs for the lower complexity bounds of distributed gradient-span algorithms applied to $\sP_{\text{\upshape SP}}$. We focus on a subclass of unconstrained SPs with $\psi_x = \psi_y = 0$, denoted $\sP_{\text{\upshape SP}}^\circ$. We assume the initial points are $\bfx^0 = \0$ and $\bfy^0 = \0$. Lower bounds established for this subclass hold for the general class $\sP_{\text{\upshape SP}}$.

Let us consider the spaces $\cE_x = \R^{n_x}$ and $\cE_y = \R^{n_y}$ equipped with the standard Euclidean inner product $\InnerProduct{\cdot}{\cdot}$ and the corresponding Euclidean norm $\Norm{\cdot}$.
The cases with general preconditioned norms $\Norm {\cdot}_x$ and $\Norm {\cdot}_y$ can be proven similarly.

\paragraph{Function subfamilies.}
To establish the overall lower bound, we decompose the general function family into three distinct subfamilies:

\begin{definition}\label{def:quadratic-subclass-x}
    Let $\sF_x = \sF (L_x, 0, 0, D_x, 0)$ denote the set of functions
    \[
    F_x(\bfx, \bfy) = \frac{1}{2} \Norm{\bfA_x \bfx - \bfb_x}^2 ,
    \]
    where the matrix $\bfA_x \in \R^{n_x \times n_x}$ satisfies $\Norm{\bfA_x}^2 \le L_x$, and the vector $\bfb_x \in \R^{n_x}$ is such that the linear system $\bfA_x \bfx = \bfb_x$ has a solution $\bfx^* \in \cE_x$ satisfying $\Norm{\bfx^*} \le D_x$.
\end{definition}

\begin{definition}\label{def:quadratic-subclass-y}
    Let $\sF_y = \sF(0, 0, L_y, 0, D_y)$ denote the set of functions
    \[
    F_y(\bfx, \bfy) = - \frac{1}{2} \Norm{\bfA_y \bfy - \bfb_y}^2 ,
    \]
    where the matrix $\bfA_y \in \R^{n_y \times n_y}$ satisfies $\Norm{\bfA_y}^2 \le L_y$, and the vector $\bfb_y \in \R^{n_y}$ is such that the linear system $\bfA_y \bfy = \bfb_y$ has a solution $\bfy^* \in \cE_y$ satisfying $\Norm{\bfy^*} \le D_y$.
\end{definition}

\begin{definition}\label{def:bilinear-subclass}
    Let $\sF_{xy} = \sF(0, L_{xy}, 0, D_x, D_y)$ denote the set of functions
    \[
    F_{xy} (\bfx, \bfy) = \InnerProduct{\bfA_{xy} \bfx - \bfb_{xy}}{\bfy} ,
    \]
    where the matrix $\bfA_{xy} \in \R^{n_y \times n_x}$ satisfies $\Norm{\bfA_{xy}} \le L_{xy}$, and the vector $\bfb_{xy} \in \R^{n_y}$ is such that the linear system $\bfA_{xy} \bfx = \bfb_{xy}$ has a solution $\bfx^* \in \cE_x$ satisfying $\Norm{\bfx^*} \le D_x$.
\end{definition}

\paragraph{Distributed oracles.}
We consider the pairs of distributed oracles $(\sO_x, \sO_y)$ that return the partial gradients of the functions in these subfamilies:
\begin{itemize}
    \item For $F_x \in \sF_x$: $\sO_x(\bfx, \bfy) = \nabla_x F_x(\bfx, \bfy) = \bfA_x^\top (\bfA_x \bfx - \bfb_x)$ and $\sO_y(\bfx, \bfy) = - \nabla_y F_x(\bfx, \bfy) = \0$.
    \item For $F_y \in \sF_y$: $\sO_x(\bfx, \bfy) = \nabla_x F_y(\bfx, \bfy) = \0$ and $\sO_y(\bfx, \bfy) = - \nabla_y F_y(\bfx, \bfy) = \bfA_y^\top (\bfA_y \bfy - \bfb_y)$.
    \item For $F_{xy} \in \sF_{xy}$: $\sO_x(\bfx, \bfy) = \nabla_x F_{xy}(\bfx, \bfy) = \bfA_{xy}^\top \bfy$ and $\sO_y(\bfx, \bfy) = - \nabla_y F_{xy}(\bfx, \bfy) = \bfb_{xy} - \bfA_{xy} \bfx$.
\end{itemize}

\paragraph{Accuracy measures.}
For any approximate solution $\bar \bfz = (\bar \bfx, \bar \bfy) \in \cE_x \times \cE_y$, the restricted primal-dual gaps over the bounded sets evaluate to:
\begin{itemize}
    \item For $F_x \in \sF_x$: $\Delta _{F_x} (\bar \bfz) = F_x(\bar \bfx, \bar \bfy) - \min _{\bfx \in \cB_x} F_x(\bfx, \bar \bfy) = \frac{1}{2} \Norm{\bfA_x \bar \bfx - \bfb_x}^2$.
    \item For $F_y \in \sF_y$: $\Delta _{F_y} (\bar \bfz) = \max _{\bfy \in \cB_y} F_y(\bar \bfx, \bfy) - F_y(\bar \bfx, \bar \bfy) = \frac{1}{2} \Norm{\bfA_y \bar \bfy - \bfb_y}^2$.
    \item For $F_{xy} \in \sF_{xy}$: $\Delta _{F_{xy}} (\bar \bfz) = \max _{\bfy \in \cB_y} \InnerProduct{\bfA_{xy} \bar \bfx - \bfb_{xy}}{\bfy} - \min _{\bfx \in \cB_x} \InnerProduct{\bfA_{xy} \bfx - \bfb_{xy}}{\bar \bfy} \ge D_y \Norm {\bfA_{xy} \bar \bfx - \bfb_{xy}}$, where the inequality holds since $\min_{\bfx \in \cB_x} \InnerProduct{\bfA_{xy} \bfx - \bfb_{xy}}{\bar \bfy} \le \InnerProduct{\bfA_{xy} \bfx^* - \bfb_{xy}}{\bar \bfy} = 0$.
\end{itemize}

\paragraph{Problem subclasses.}
These constructions yield three distinct problem subclasses:
\begin{itemize}
    \item $\sP_x = \bigl( \sF_x, \sO_x, \sO_y, \epsilon \bigr)$,
    \item $\sP_y = \bigl( \sF_y, \sO_x, \sO_y, \epsilon \bigr)$, and
    \item $\sP_{xy} = \bigl( \sF_{xy}, \sO_{x}, \sO_{y}, \epsilon \bigr)$.
\end{itemize}
The worst-case complexity for general SPs is bounded below by the maximum complexity among these three subclasses.

For the rest of this section, we present our lower bound proofs by establishing a connection from distributed SPs to convex minimization. In \cref{sec:lower-bound-worst-function}, we recall and properly rescale Nesterov's construction for unconstrained convex optimization. We analyze the three problem subclasses separately in \cref{sec:lower-bound-quadratic-subclasses,sec:lower-bound-bilinear-subclass}, mapping the lower bound for each subclass to this core convex minimization problem. Finally, we combine the results to establish the overall lower bound for distributed SPs in \cref{sec:proof-of-main-lower-bound-theorem}.

\subsubsection{The ``worst function in the world'' (with proper rescaling)}
\label{sec:lower-bound-worst-function}

While the original formulation~(Theorem~2.1.7 by~\cite{nesterov2004introductory}) establishes a lower bound with respect to an arbitrary initial distance $\Norm{\bfv^0 - \bfv^*}$, this is insufficient for the lower bound proof in this paper. We are working on a lower bound analysis with a more specific problem subclass, where the algorithms may be designed with prior knowledge of a bound $D$ on the distance to the optimum. To this end, we present a refined version of Nesterov's proposition by properly rescaling Nesterov's original construction to ensure that the distance to the optimum is bounded by $D$, while maintaining the identical lower bound on the function value.

For any given matrix $\bfA \in \R^{m \times n}$ and vector $\bfb \in \R^m$, we define the $k$-th Krylov subspace as follows:
\[
\cH^k (\bfA, \bfb) \triangleq \spn \bigl\{ \bfA ^\top \bfb, (\bfA ^\top \bfA) \bfA ^\top \bfb, \ldots, (\bfA ^\top \bfA)^{k-1} \bfA ^\top \bfb \bigr\}.
\]

\begin{proposition}\label{proposition:worst-cast-construction-bounds}
Let $L > 0$, $D > 0$, and integer $1 \le k \le \frac{\min\{m-1, n\} - 1}{2}$. Then, there exist a matrix $\bfA = \bfA(L, k) \in \R^{m \times n}$ with $\Norm {\bfA} \le L$ and a vector $\bfb = \bfb(L, D, k) \in \R^m$, such that
\[
\min_{\bfv \in \cH^k(\bfA, \bfb)}\ \frac{1}{2} \Norm{\bfA \bfv - \bfb}^2 \ge \frac{3 L^2 D^2}{32 (k+1)^2},
\]
and the linear system $\bfA \bfv = \bfb$ has a solution $\bfv^* \in \R^n$ satisfying $\Norm{\bfv^*} \le D$.
\end{proposition}

\begin{proof}
    Let $p = 2k + 1$. By the condition $k \le \frac{\min\{m-1, n\} - 1}{2}$, we have $p+1 \le m$ and $p \le n$.
    Let $\bfM_p \in \R^{p \times p}$ be the symmetric tridiagonal matrix defined as:
    \[
    \bfM_p = \begin{bmatrix}
        2 & -1 & 0 & \cdots & 0 \\
        -1 & 2 & -1 & \cdots & 0 \\
        0 & -1 & 2 & \ddots & \vdots \\
        \vdots & \vdots & \ddots & \ddots & -1 \\
        0 & 0 & \cdots & -1 & 2
    \end{bmatrix} .
    \]
    Let $\bfB_p \in \R^{(p+1) \times p}$ be the matrix such that $\bfB_p^\top \bfB_p = \bfM_p$, defined as:
    \[
    \bfB_p = \begin{bmatrix}
        1 & 0 & 0 & \cdots & 0 \\
        -1 & 1 & 0 & \cdots & 0 \\
        0 & -1 & 1 & \ddots & \vdots \\
        \vdots & \vdots & \ddots & \ddots & 0 \\
        0 & 0 & \cdots & -1 & 1 \\
        0 & 0 & \cdots & 0 & -1
    \end{bmatrix} .
    \]
    We define the matrix $\bfA \in \R^{m \times n}$ as the block matrix:
    \[
    \bfA = \frac{L}{2} \begin{bmatrix}
        \bfB_p & \0_{(p+1) \times (n-p)} \\
        \0_{(m-p-1) \times p} & \0_{(m-p-1) \times (n-p)}
    \end{bmatrix} .
    \]
    The matrix $\bfA^\top \bfA \in \R^{n \times n}$ is given by the block-diagonal matrix:
    \[
    \bfA^\top \bfA = \frac{L^2}{4} \begin{bmatrix}
        \bfM_p & \0 \\
        \0 & \0
    \end{bmatrix} .
    \]
    The spectral norm of $\bfA$ is bounded as $\Norm{\bfA} = \sqrt{\lambda_{\max}(\bfA^\top \bfA)} = \frac{L}{2} \sqrt{\lambda_{\max}(\bfM_p)} \le L$.

    Let $\gamma = D \sqrt{\frac{6(p+1)}{p(2p+1)}}$. We define the vector $\bfu \in \R^{p+1}$ by its coordinates:
    \[
    u_1 = \frac{p}{p+1} , \quad \text{and} \quad u_i = -\frac{1}{p+1} \quad \text{for } 2 \le i \le p+1 .
    \]
    By the structure of $\bfB_p^\top$, we have $\bfB_p^\top \bfu = \bfe_1^{(p)} \in \R^p$.
    Its squared norm evaluates to:
    \[
    \Norm{\bfu}^2 = \Bigl( \frac{p}{p+1} \Bigr)^2 + p \Bigl( -\frac{1}{p+1} \Bigr)^2 = \frac{p^2 + p}{(p+1)^2} = \frac{p}{p+1} .
    \]
    We define $\bfb \in \R^m$ as the block vector:
    \[
    \bfb = \gamma \frac{L}{2} \begin{bmatrix}
        \bfu \\
        \0_{m-p-1}
    \end{bmatrix} .
    \]
    The linear system $\bfA \bfv = \bfb$ has a solution $\bfv^* \in \R^n$ given by the block vector:
    \[
    \bfv^* = \gamma \begin{bmatrix}
        \bfM_p^{-1} \bfe_1^{(p)} \\
        \0_{n-p}
    \end{bmatrix} .
    \]
    By the structure of $\bfM_p^{-1}$, the coordinates of $\bfv^*$ are $v_i^* = \gamma \frac{p+1-i}{p+1}$ for $1 \le i \le p$, and $0$ otherwise. Its squared norm evaluates to:
    \[
    \Norm{\bfv^*}^2 = \frac{\gamma^2}{(p+1)^2} \sum_{j=1}^p j^2 = \gamma^2 \frac{p(2p+1)}{6(p+1)} = D^2.
    \]

    The Krylov subspace is defined as:
    \[
    \cH^k(\bfA, \bfb) = \spn\{\bfA^\top \bfb, (\bfA^\top \bfA) \bfA^\top \bfb, \dots, (\bfA^\top \bfA)^{k-1} \bfA^\top \bfb\}.
    \]
    Since $\bfB_p^\top \bfu = \bfe_1^{(p)}$, the initial vector evaluates to $\bfA^\top \bfb = \gamma \frac{L^2}{4} \bfe_1^{(n)}$. Successive multiplication of $\bfe_1^{(n)}$ by $\bfA^\top \bfA$ expands the non-zero support by one standard basis vector at a time. Thus, the subspace spans the first $k$ standard basis vectors:
    \[
    \cH^k(\bfA, \bfb) = \spn\{\bfe_1^{(n)}, \bfe_2^{(n)}, \dots, \bfe_k^{(n)}\}.
    \]

    For any $\bfv \in \cH^k(\bfA, \bfb)$, its non-zero support is confined to the first $k$ coordinates. Let $\bfv_k \in \R^k$ denote these first $k$ coordinates, such that the first $p$ coordinates of $\bfv$ are $\bfv_p = (\bfv_k^\top, \0_{p-k}^\top)^\top$.
    Let $\bfM_k \in \R^{k \times k}$ be the leading principal submatrix of $\bfM_p$. The squared residual norm for $\bfv \in \cH^k(\bfA, \bfb)$ evaluates to:
    \[
    \begin{aligned}
    \frac{1}{2} \Norm{\bfA \bfv - \bfb}^2
    &= \frac{L^2}{8} \Norm{ \bfB_p \bfv_p - \gamma \bfu }^2 \\
    &= \frac{L^2}{8} \Bigl( \langle \bfv_p, \bfB_p^\top \bfB_p \bfv_p \rangle - 2\gamma \langle \bfv_p, \bfB_p^\top \bfu \rangle + \gamma^2 \langle \bfu, \bfu \rangle \Bigr) \\
    &= \frac{L^2}{8} \Bigl( \langle \bfv_k, \bfM_k \bfv_k \rangle - 2\gamma \langle \bfv_k, \bfe_1^{(k)} \rangle + \gamma^2 \frac{p}{p+1} \Bigr) .
    \end{aligned}
    \]
    Denoting $v_1 = \langle \bfv_k, \bfe_1^{(k)} \rangle$, the norm simplifies to:
    \[
    \frac{1}{2} \Norm{\bfA \bfv - \bfb}^2 = \frac{L^2}{8} \Bigl( \langle \bfv_k, \bfM_k \bfv_k \rangle - 2\gamma v_1 + \gamma^2 \frac{p}{p+1} \Bigr) .
    \]

    Minimizing this residual norm over $\bfv_k \in \R^k$ yields the optimal solution $\bfv_k^* = \gamma \bfM_k^{-1} \bfe_1^{(k)}$, equivalently characterized by $\bfM_k \bfv_k^* = \gamma \bfe_1^{(k)}$. By the structure of $\bfM_k^{-1}$, the coordinates of $\bfv_k^*$ are $v_{k, i}^* = \gamma \frac{k+1-i}{k+1}$ for $1 \le i \le k$.
    Substituting $\bfv_k = \bfv_k^*$ into the $\bfv_k$-dependent terms (and noting $v_1 = \langle \bfv_k^*, \bfe_1^{(k)} \rangle = v_{k, 1}^*$), the minimum value evaluates to:
    \[
    \begin{aligned}
    \bigl( \langle \bfv_k, \bfM_k \bfv_k \rangle - 2\gamma v_1 \bigr) \Big|_{\bfv_k = \bfv_k^*}
    &= \langle \bfv_k^*, \bfM_k \bfv_k^* \rangle - 2\gamma v_{k, 1}^* \\
    &= \langle \bfv_k^*, \gamma \bfe_1^{(k)} \rangle - 2\gamma v_{k, 1}^* \\
    &= \gamma v_{k, 1}^* - 2\gamma v_{k, 1}^*
    = -\gamma v_{k, 1}^* = -\gamma^2 \frac{k}{k+1} .
    \end{aligned}
    \]

    Substituting these values gives the minimum residual norm over the Krylov subspace:
    \[
    \begin{aligned}
    \min_{\bfv \in \cH^k(\bfA, \bfb)} \frac{1}{2} \Norm{\bfA \bfv - \bfb}^2
    &= \frac{L^2 \gamma^2}{8} \Bigl( \frac{p}{p+1} - \frac{k}{k+1} \Bigr) \\
    &= \frac{L^2 \gamma^2}{8} \Bigl( \frac{2k+1}{2k+2} - \frac{k}{k+1} \Bigr) \\
    &= \frac{L^2 \gamma^2}{16(k+1)} \\
    &= \frac{L^2 D^2}{16(k+1)} \frac{6(2k+2)}{(2k+1)(4k+3)} \\
    &= \frac{3 L^2 D^2}{4(8k^2 + 10k + 3)}.
    \end{aligned}
    \]
    Since $8k^2 + 10k + 3 \le 8(k+1)^2$, this residual norm is lower bounded by $\frac{3 L^2 D^2}{32(k+1)^2}$.
\end{proof}

\subsubsection{Quadratic subclasses \texorpdfstring{$\sP_x$ and $\sP_y$}{Px and Py}}
\label{sec:lower-bound-quadratic-subclasses}

In this section, we apply \cref{proposition:worst-cast-construction-bounds} to establish lower bounds for the quadratic subclasses $\sP_x$ and $\sP_y$. It should be noted that to establish lower bounds for quadratic subclasses, only the assumptions on local variables (\cref{assum:local-variables}) and candidate solutions (\cref{assum:generated-solutions}) are used.

For any instance $F_x \in \sP_x$ of the form $F_x(\bfx)$, the partial gradient with respect to $\bfy$ is identically zero, and the gradient $\nabla_x F_x(\bfx)$ depends only on the local variable $\bfx$. As a result, the communication between agents provides no additional information, and the restriction on Agent~$x$ reduces to the standard gradient-span condition for single-node unconstrained convex optimization. That is, for this subclass, \cref{assum:local-variables,assum:generated-solutions} become:
\begin{align*}
    \bfx^{t,l} &\in \bfx^0 + \bfP_x^{-1} \spn \bigl\{ \nabla_x F_x ( \bfx^{i,j} ) \mid 0 \le i < t,\ 0 \le j \le \tau_x^i - 1 \text{ or } i = t,\ 0 \le j \le l - 1 \bigr\} , \\
    \bar \bfx^{t+1} &\in \bfx^0 + \bfP_x^{-1} \spn \bigl\{ \nabla_x F_x ( \bfx^{i,j} ) \mid 0 \le i \le t,\ 0 \le j \le \tau_x^i - 1 \bigr\} .
\end{align*}

Notice that the gradient span sequences are naturally confined to the standard Krylov subspaces. Specifically, for $\bfA_x \in \R^{n_x \times n_x}$ and $\bfb_x \in \R^{n_x}$, the subspace satisfies the algebraic progression:
\[
(\bfA_x^\top \bfA_x) \cH^m (\bfA_x, \bfb_x) \subseteq \cH^{m+1} (\bfA_x, \bfb_x)
\quad \text{and} \quad
\bfA_x^\top \bfb_x \in \cH^{m+1} (\bfA_x, \bfb_x) .
\]

\begin{proposition}\label{prop:krylov-bound-quadratic-x}
    Let $\sM$ be a distributed method satisfying \cref{assum:local-variables} and \cref{assum:generated-solutions}, operating on an instance $F_x \in \sP_x$ of the form $F_x(\bfx) = \frac{1}{2} \Norm{\bfA_x \bfx - \bfb_x}^2$. Let $\bar \bfx$ be the candidate solution generated by Agent~$x$ after evaluating $N_x$ partial gradients. Then, $\bar \bfx \in \cH^{N_x} (\bfA_x, \bfb_x)$.
\end{proposition}

\begin{proof}
    For $F_x \in \sP_x$, we have $\psi_x = \psi_y = 0$. Let $\bfz_1, \dots, \bfz_{N_x}$ be the sequence of query points evaluated by Agent~$x$, where $\bfz_i = (\bfx_i, \bfy_i)$. Since $\bfx^0 = \0$, \cref{assum:local-variables} and \cref{assum:generated-solutions} require that each local query point $\bfx_{m+1}$ and the candidate solution $\bar \bfx$ reside in the span of historical gradients. Let $S_m \triangleq \spn \{ \nabla_x F_x(\bfz_i) \mid 1 \le i \le m \} \subseteq \cE_x$. We have $\bfx_{m+1} \in S_m$ and $\bar \bfx \in S_{N_x}$.

    We show $S_m \subseteq \cH^m (\bfA_x, \bfb_x)$ by induction. The base case $m=0$ holds since $S_0 = \{\0\} \subseteq \cH^0 (\bfA_x, \bfb_x)$.

    Assume $S_m \subseteq \cH^m (\bfA_x, \bfb_x)$ for some $m \ge 0$. The gradient at $\bfx_{m+1}$ evaluates to $\nabla_x F_x(\bfx_{m+1}) = \bfA_x^\top \bfA_x \bfx_{m+1} - \bfA_x^\top \bfb_x$. Since $\bfx_{m+1} \in S_m \subseteq \cH^m (\bfA_x, \bfb_x)$, applying the algebraic progression properties yields:
    \[
    \nabla_x F_x(\bfx_{m+1}) \in (\bfA_x^\top \bfA_x) \cH^m (\bfA_x, \bfb_x) - \bfA_x^\top \bfb_x \subseteq \cH^{m+1} (\bfA_x, \bfb_x) .
    \]
    Thus, $S_{m+1} = S_m + \spn \{ \nabla_x F_x(\bfx_{m+1}) \} \subseteq \cH^{m+1} (\bfA_x, \bfb_x)$.

    By induction, $S_{N_x} \subseteq \cH^{N_x} (\bfA_x, \bfb_x)$, which implies $\bar \bfx \in \cH^{N_x} (\bfA_x, \bfb_x)$.
\end{proof}

Analogously, for any instance $F_y \in \sP_y$ evaluated by Agent~$y$, of the form $F_y(\bfx, \bfy) = - \frac{1}{2} \Norm{\bfA_y \bfy - \bfb_y}^2$, the candidate solution generated after $N_y$ queries satisfies $\bar \bfy \in \cH^{N_y}(\bfA_y, \bfb_y)$.

\begin{theorem}
\label{thm:lower-bound-quadratic-subclasses}
    Let $\sM$ be a distributed method satisfying \cref{assum:local-variables} and \cref{assum:generated-solutions} for problem class $\sP_{\text{\upshape SP}}$, where $n_x \ge 2 \sqrt{\frac{3 L_x D_x^2}{32 \epsilon}} + 2$ and $n_y \ge 2 \sqrt{\frac{3 L_y D_y^2}{32 \epsilon}} + 2$.
    Then, we have
    \[
    N^{\sM}_{\sP_x} \ge c_x \biggl( \sqrt{\frac{3 L_x D_x^2}{32 \epsilon}} - 1 \biggr) ,
    \]
    and analogously,
    \[
    N^{\sM}_{\sP_y} \ge c_y \biggl( \sqrt{\frac{3 L_y D_y^2}{32 \epsilon}} - 1 \biggr) .
    \]
\end{theorem}

\begin{proof}
    We prove the bound for $\sP_x$. Let $K = \lfloor \sqrt{\frac{3 L_x D_x^2}{32 \epsilon}} \rfloor$. Since $n_x \ge 2 K + 2$, we apply \cref{proposition:worst-cast-construction-bounds} to obtain $\bfA_K = \bfA(\sqrt{L_x}, K)$ and $\bfb_K = \bfb(\sqrt{L_x}, D_x, K)$ constructed in $\R^{n_x \times n_x}$ and $\R^{n_x}$.

    We set $\bfA_x = \bfA_K$, $\bfb_x = \bfb_K$, and define $F_K(\bfx, \bfy) = \frac{1}{2} \Norm{\bfA_x \bfx - \bfb_x}^2$.
    By \cref{proposition:worst-cast-construction-bounds}, $\Norm{\bfA_x}^2 \le L_x$, and the linear system $\bfA_x \bfx = \bfb_x$ has a solution $\bfx^*$ satisfying $\Norm{\bfx^*} \le D_x$. Hence, $F_K \in \sP_x$.

    When $\sM$ is applied to $P_K$, suppose it generates an $\epsilon$-saddle point $\bar \bfz$ utilizing $N_x$ queries to $\sO_x$. If $N_x \ge K$, the bound $N_x \ge \sqrt{\frac{3 L_x D_x^2}{32 \epsilon}} - 1$ holds. If $N_x < K$, \cref{prop:krylov-bound-quadratic-x} implies $\bar{\bfx} \in \cH^{N_x}(\bfA_x, \bfb_x)$.

    Bounding the restricted duality gap via \cref{proposition:worst-cast-construction-bounds} yields:
    \[
    \epsilon \ge \Delta_{F_K}(\bar \bfz)
    = \frac{1}{2} \Norm{\bfA_x \bar{\bfx} - \bfb_x}^2
    \ge \min_{\bfv \in \cH^{N_x} (\bfA_x, \bfb_x)} \frac{1}{2} \Norm{\bfA_x \bfv - \bfb_x}^2
    \ge \frac{3 L_x D_x^2}{32 (N_x + 1)^2} .
    \]
    Rearranging gives $N_x \ge \sqrt{\frac{3 L_x D_x^2}{32 \epsilon}} - 1$.
    Multiplying by $c_x$ yields $N^{\sM}_{\sP_x} \ge c_x \bigl( \sqrt{\frac{3 L_x D_x^2}{32 \epsilon}} - 1 \bigr)$.

    The proof for $\sP_y$ is symmetric, setting $\bfA_y = \bfA(\sqrt{L_y}, K)$ and $\bfb_y = \bfb(\sqrt{L_y}, D_y, K)$ for $K = \lfloor \sqrt{\frac{3 L_y D_y^2}{32 \epsilon}} \rfloor$.
\end{proof}

\subsubsection{Bilinear subclass \texorpdfstring{$\sP_{xy}$}{Pxy}}
\label{sec:lower-bound-bilinear-subclass}

In this section, we establish the lower bounds for the bilinear subclass $\sP_{xy}$. It should be noted that to establish lower bounds for the bilinear subclass, only the assumptions on remote variables (\cref{assum:remote-variables}) and candidate solutions (\cref{assum:generated-solutions}) are used.

Consider any function from $\sP_{xy}$ of the form $F(\bfx, \bfy) = \InnerProduct{\bfA \bfx - \bfb}{\bfy}$, where $\bfA \in \R^{n_y \times n_x}$ and $\bfb \in \R^{n_y}$. The partial gradients are $\nabla_x F(\bfx, \bfy) = \bfA^\top \bfy$ and $\nabla_y F(\bfx, \bfy) = \bfA \bfx - \bfb$. For any $\bar \bfz = (\bar \bfx, \bar \bfy) \in \cE_x \times \cE_y$, the restricted primal-dual gap evaluates to:
\[
\Delta_F (\bar \bfz) \ge \max _{\bfy \in \cB_y}\ \InnerProduct {\bfA \bar \bfx - \bfb} {\bfy} = D_y \Norm{\bfA \bar \bfx - \bfb} .
\]

We define the coupled Krylov subspaces in $\cE_y$ and $\cE_x$ as:
\[
\cH^k_y(\bfA, \bfb) \triangleq \spn \bigl\{ \bfb, (\bfA \bfA^\top) \bfb, \dots, (\bfA \bfA^\top)^{k-1} \bfb \bigr\} ,
\]
\[
\cH^k_x(\bfA, \bfb) \triangleq \spn \bigl\{ \bfA^\top \bfb, (\bfA^\top \bfA) \bfA^\top \bfb, \dots, (\bfA^\top \bfA)^{k-1} \bfA^\top \bfb \bigr\} \equiv \cH^k(\bfA, \bfb) .
\]
By definition, these subspaces satisfy the alternating properties:
\[
\bfA \cH^m_x(\bfA, \bfb) + \spn\{\bfb\} = \cH^{m+1}_y(\bfA, \bfb)
\quad \text{and} \quad
\bfA^\top \cH^m_y(\bfA, \bfb) = \cH^m_x(\bfA, \bfb) \subseteq \cH^{m+1}_x(\bfA, \bfb) .
\]

\begin{proposition}\label{prop:krylov-bound-T}
    Let $\sM$ be a distributed method satisfying \cref{assum:remote-variables} and \cref{assum:generated-solutions}, operating on $F(\bfx, \bfy) = \InnerProduct{\bfA \bfx - \bfb}{\bfy}$. Let $\bar{\bfx}$ be the candidate solution generated after $T$ communication rounds. Then, $\bar{\bfx} \in \cH^{\lceil (T-1)/2 \rceil}_x(\bfA, \bfb)$.
\end{proposition}

\begin{proof}
    Let $S_x^t \triangleq \spn \{ \nabla_x F(\bfz) \mid \bfz \in Z_x^t \} \subseteq \cE_x$ and $S_y^t \triangleq \spn \{ -\nabla_y F(\bfz) \mid \bfz \in Z_y^t \} \subseteq \cE_y$.
    We show by induction that $S_x^t \subseteq \cH^{\lceil t/2 \rceil}_x(\bfA, \bfb)$ and $S_y^t \subseteq \cH^{\lfloor t/2 \rfloor + 1}_y(\bfA, \bfb)$.

    The base case $t=-1$ holds since $S_x^{-1} = \{\0\} \subseteq \cH^0_x(\bfA, \bfb)$ and $S_y^{-1} = \{\0\} \subseteq \cH^0_y(\bfA, \bfb)$.

    Assume the claim holds for round $t-1$. In round $t$, Agent~$x$ queries points using remote variables $\hat{\bfy}$. By \cref{assum:remote-variables}, $\hat{\bfy} \in S_y^{t-1}$. Applying the alternating properties, the evaluated gradient satisfies:
    \[
    \nabla_x F(\bfx, \hat{\bfy}) = \bfA^\top \hat{\bfy} \in \bfA^\top \cH^{\lfloor (t-1)/2 \rfloor + 1}_y(\bfA, \bfb) = \cH^{\lfloor (t-1)/2 \rfloor + 1}_x(\bfA, \bfb) = \cH^{\lceil t/2 \rceil}_x(\bfA, \bfb) .
    \]
    Thus, $S_x^t \subseteq \cH^{\lceil t/2 \rceil}_x(\bfA, \bfb)$. Similarly, Agent~$y$ queries points using remote variables $\hat{\bfx}$. By \cref{assum:remote-variables}, $\hat{\bfx} \in S_x^{t-1}$. The evaluated gradient satisfies:
    \[
    -\nabla_y F(\hat{\bfx}, \bfy) = \bfb - \bfA \hat \bfx \in \spn\{\bfb\} + \bfA \cH^{\lceil (t-1)/2 \rceil}_x(\bfA, \bfb) = \cH^{\lceil (t-1)/2 \rceil + 1}_y(\bfA, \bfb) = \cH^{\lfloor t/2 \rfloor + 1}_y(\bfA, \bfb) .
    \]
    Thus, $S_y^t \subseteq \cH^{\lfloor t/2 \rfloor + 1}_y(\bfA, \bfb)$. By induction, the claim holds for all $t$. By \cref{assum:generated-solutions}, the candidate solution satisfies $\bar{\bfx} \in S_x^{T-1} \subseteq \cH^{\lceil (T-1)/2 \rceil}_x(\bfA, \bfb)$.
\end{proof}

\begin{theorem}
\label{thm:lower-bound-bilinear}
    Let $\sM$ be a distributed method satisfying \cref{assum:remote-variables} and \cref{assum:generated-solutions} for problem class $\sP_{xy}$, where $\min \{ n_x, n_y \} \ge \frac{2 L_{xy} D_x D_y}{3 \epsilon} + 2$. Then, we have
    \[
    T^{\sM}_{\sP_{xy}} \ge \frac{2 L_{xy} D_x D_y}{3 \epsilon} - 2 .
    \]
\end{theorem}

\begin{proof}
    Let $K = \lfloor \frac{L_{xy} D_x D_y}{3 \epsilon} \rfloor$. Since $n_x, n_y \ge 2 K + 2$, we apply \cref{proposition:worst-cast-construction-bounds} to obtain $\bfA_K = \bfA(L_{xy}, K) \in \R^{n_y \times n_x}$ and $\bfb_K = \bfb(L_{xy}, D_x, K) \in \R^{n_y}$.

    We set $\bfA_{xy} = \bfA_K$ and $\bfb_{xy} = \bfb_K$, and define $F_K(\bfx, \bfy) = \InnerProduct{\bfA_{xy} \bfx - \bfb_{xy}}{\bfy}$.
    By \cref{proposition:worst-cast-construction-bounds}, $\Norm{\bfA_{xy}} = \Norm{\bfA_K} \le L_{xy}$. The linear system $\bfA_{xy} \bfx = \bfb_{xy}$ has a solution $\bfx^*$ satisfying $\Norm{\bfx^*} \le D_x$. Thus, $F_K \in \sP_{xy}$.

    When $\sM$ is applied to $P_K$, suppose it generates an $\epsilon$-saddle point $(\bar{\bfx}, \bar{\bfy})$ after $T$ communication rounds. If $T \ge 2K$, then $T \ge \frac{2 L_{xy} D_x D_y}{3 \epsilon} - 2$ holds. If $T < 2K$, \cref{prop:krylov-bound-T} implies $\bar{\bfx} \in \cH^k_x(\bfA_{xy}, \bfb_{xy})$, where $k = \lceil (T-1)/2 \rceil$.

    Since $\cH^k_x(\bfA_{xy}, \bfb_{xy}) \equiv \cH^k(\bfA_K, \bfb_K)$, the restricted duality gap on $P_K$ evaluates to:
    \[
    \Delta_{F_K}(\bar{\bfx}, \bar{\bfy}) \ge D_y \Norm{\bfA_K \bar \bfx - \bfb_K} .
    \]
    Applying \cref{proposition:worst-cast-construction-bounds} yields:
    \[
    \epsilon \ge \Delta_{F_K}(\bar{\bfx}, \bar{\bfy}) \ge D_y \min_{\bfv \in \cH^k(\bfA_K, \bfb_K)} \Norm{\bfA_K \bfv - \bfb_K} \ge D_y \sqrt{ \frac{3 L_{xy}^2 D_x^2}{2 (8k^2 + 10k + 3)} } \ge \frac{L_{xy} D_x D_y}{3 (k+1)} .
    \]
    Rearranging gives $k \ge \frac{L_{xy} D_x D_y}{3 \epsilon} - 1$.
    Because $k = \lceil (T-1)/2 \rceil$, we have $T \ge 2k \ge \frac{2 L_{xy} D_x D_y}{3 \epsilon} - 2$.
\end{proof}

The oracle lower bounds can be proved using identical subspace confinement arguments. Rather than repeating the proof, we provide the main proposition and theorem here.

\begin{proposition}\label{prop:query-points-and-approximate-solutions-live-in-Krylov-subspaces}
    Let $\sM$ be a distributed method satisfying \cref{assum:remote-variables} and \cref{assum:generated-solutions}, operating on $F(\bfx, \bfy) = \InnerProduct{\bfA \bfx - \bfb}{\bfy}$. Let $(\bar \bfx, \bar \bfy)$ be the solution generated after $N_x$ and $N_y$ queries to $\sO_x$ and $\sO_y$, respectively. Then, $\bar{\bfx} \in \cH^k_x (\bfA, \bfb)$, where $k = \min(N_x, N_y)$.
\end{proposition}

\begin{theorem}
    Let $\sM$ be a distributed method satisfying \cref{assum:remote-variables} and \cref{assum:generated-solutions} for problem class $\sP_{xy}$, where $\min \{ n_x, n_y \} \ge \frac{2 L_{xy} D_x D_y}{3 \epsilon} + 2$. Then, we have
    \[
    N^{\sM}_{\sP_{xy}} \ge (c_x + c_y) \biggl( \frac{L_{xy} D_x D_y}{3 \epsilon} - 1 \biggr) .
    \]
\end{theorem}

\subsubsection{Lower complexity bounds}
\label{sec:proof-of-main-lower-bound-theorem}

Finally, we derive the lower bound for $\sP_{\text{\upshape SP}}$ by assembling the lower bounds obtained from the three subclasses.

\begin{theorem}
\label[theorem]{theorem:lower_complexity-bounds}
    Let $\sM$ be a distributed gradient-span algorithm for $\sP _{\text{\upshape SP}}$.
    Suppose the dimensions of $\cE_x$ and $\cE_y$ are sufficiently large such that $n_x \ge \frac{2 L_{xy} D_x D_y}{3\epsilon} + \sqrt{\frac {3 L_x D_x^2} {8 \epsilon}} + 2$ and $n_y \ge \frac{2 L_{xy} D_x D_y}{3\epsilon} + \sqrt{\frac {3 L_y D_y^2} {8 \epsilon}} + 2$.
    We have:
    \[
    \begin{gathered}
    T ^{\sM} _{{\sP}_{\text{\upshape SP}}} \ge \frac{2 L_{xy} D_x D_y}{3 \epsilon} - 2 , \\
    N ^ {\sM} _{{\sP} _{\text{\upshape SP}}} \ge \frac{c_x + c_y}{9} \frac{L_{xy} D_x D_y}{\epsilon} + \frac {c_x} 3 \sqrt{\frac{3 L_x D_x^2}{32 \epsilon}} + \frac {c_y} 3 \sqrt{\frac{3 L_y D_y^2}{32 \epsilon}} - \frac{2c_x + 2c_y}{3}.
    \end{gathered}
    \]
\end{theorem}

\begin{proof}[Proof of \cref{theorem:lower_complexity-bounds}]
    The general class of convex-concave saddle point problems $\sP_{\text{\upshape SP}}$ contains the three unregularized subclasses constructed in the previous sections: the $\bfx$-quadratic subclass $\sP_x$, the $\bfy$-quadratic subclass $\sP_y$, and the bilinear subclass $\sP_{xy}$. The worst-case complexity for an algorithm operating over the entire class $\sP_{\text{\upshape SP}}$ is bounded from below by the maximum of the complexities required for these individual subclasses.

    Because a gradient-span algorithm satisfies all three assumptions (\cref{assum:local-variables,assum:remote-variables,assum:generated-solutions}), we can assemble these results. By \cref{thm:lower-bound-bilinear}, the communication complexity over the class is bounded by the communication complexity of the bilinear subclass:
    \[
    T^{\sM}_{\sP_{\text{\upshape SP}}} \ge T^{\sM}_{\sP_{xy}} \ge \frac{2 L_{xy} D_x D_y}{3 \epsilon} - 2.
    \]

    For the computational complexity, we combine the independent lower bounds established in \cref{thm:lower-bound-quadratic-subclasses} and \cref{thm:lower-bound-bilinear}. The total computational complexity is bounded by the maximum of the three individual requirements:
    \[
    N^{\sM}_{\sP_{\text{\upshape SP}}} \ge \max \Bigl\{ N^{\sM}_{\sP_{xy}}, N^{\sM}_{\sP_x}, N^{\sM}_{\sP_y} \Bigr\}.
    \]
    Using the algebraic property $\max\{a, b, c\} \ge \frac{1}{3}(a + b + c)$, we obtain the lower bound:
    \[
    \begin{aligned}
    N^{\sM}_{\sP_{\text{\upshape SP}}}
    &\ge \frac 1 3 \Bigl[ N^{\sM}_{\sP_{xy}} + N^{\sM}_{\sP_x} + N^{\sM}_{\sP_y} \Bigr] \\
    &\ge \frac 1 3 \Biggl[ (c_x + c_y) \biggl( \frac{L_{xy} D_x D_y}{3 \epsilon} - 1 \biggr) + c_x \biggl( \sqrt{\frac{3 L_x D_x^2}{32 \epsilon}} - 1 \biggr) + c_y \biggl( \sqrt{\frac{3 L_y D_y^2}{32 \epsilon}} - 1 \biggr) \Biggr] \\
    &= \frac{c_x + c_y}{9} \frac{L_{xy} D_x D_y}{\epsilon} + \frac {c_x} 3 \sqrt{\frac{3 L_x D_x^2}{32 \epsilon}} + \frac {c_y} 3 \sqrt{\frac{3 L_y D_y^2}{32 \epsilon}} - \frac{2c_x + 2c_y}{3}.
    \end{aligned}
    \]
    This establishes the stated lower bounds for the general problem class and concludes the proof.
\end{proof}

\begin{remark}[Communication optimality]
\label{remark:gap-from-lower-bound}
    \Cref{theorem:lower_complexity-bounds} confirms that the communication lower bound for distributed SPs depends only on the cross-coupled conditioning.
    More importantly, the communication lower bound perfectly matches our upper bound proven in \cref{thm:DROMsp-communication-complexity} up to a constant.
    Therefore, we conclude that our \DMsp is a communication-optimal algorithm within the gradient-span framework.
\end{remark}

We note, however, that a gap remains between the achieved oracle costs and the theoretical oracle lower bound established in \cref{theorem:lower_complexity-bounds}. Yet this is a known open question even for non-distributed SPs.

\section{Variational inequality problems with distributed oracles}
\label{sec:monotone-block-VIPs}

\paragraph{Motivation.}
Thus far, we have studied SPs, which naturally model two-player zero-sum games. To capture more complex multiagent interactions~(such as equilibrium computation in multiplayer general-sum games, network routing, and multiagent reinforcement learning), we extend our algorithmic framework to the broader class of monotone Variational Inequality Problems (VIPs) with separable composite terms and distributed oracles. We briefly outline the problem class and our results here, deferring the detailed presentation to \cref{sec:monotone-block-VIPs-appendix}.

\paragraph{Problem class.}
We consider a distributed multiagent setting with a star communication network over a product space $\cE = \cE_1 \times \cdots \times \cE_K$, where a joint decision variable is partitioned among $K$ autonomous agents as $\bfz = (\bfz_1, \dots, \bfz_K)$. Let us consider the problem class $\sP _{\text{\upshape VIP}}$ as follows:
\begin{itemize}
    \item \textbf{Operators and local components:} The problem is governed by a joint monotone operator $V(\bfz) = (V_1(\bfz), \cdots, V_K(\bfz))$ and a separable local composite function $\psi(\bfz) = \sum_{i=1}^K \psi_i(\bfz_i)$ defined on $Q = \dom \psi_1 \times \cdots \times \dom \psi_K$. We assume the solution set is bounded by local distance parameters $D_i > 0$ for each agent. Furthermore, the operator satisfies block-wise Lipschitz continuity: for any fixed $\bfz_{-j}$, the mapping $V_i(\bfz_j; \bfz_{-j})$ is $L_{ij}$-Lipschitz continuous with respect to $\bfz_j$.
    \item \textbf{Distributed oracles:} Each Agent~$i \in [K]$ controls its local variable $\bfz_i \in \dom \psi_i$, has access to its private function $\psi_i$ and a partial oracle $\sO_i(\bfz) = V_i(\bfz)$.
    \item \textbf{Accuracy measure:} The goal is to find an $\epsilon$-approximate solution $\bar \bfz \in Q$ such that the restricted gap $\Delta (\bar \bfz) \triangleq \sup _{\bfz \in \cB \cap Q} [\InnerProduct { V (\bfz) }{ \bar \bfz - \bfz } + \psi (\bar \bfz) - \psi (\bfz)]$ satisfies $\Delta (\bar \bfz) \le \epsilon$, where the bounded domain $\cB$ is defined by the initial point $\bfz^0$ and distance parameters $D_i$, $i \in [K]$.
\end{itemize}

\paragraph{Conditionings.}
Let $\bar L_{ij} \triangleq \max \{ L_{ij}, L_{ji} \}$.
Similar to SPs, we hereby define the \emph{cross-coupled conditioning}, denoted by $\sum_{i \in [K]} A_i$ with $A_i \triangleq D_i \sum _{ j \neq i } \bar L_{ij} D_j$, that quantifies the cross-dependencies between the agents over the network. The \emph{diagonal conditioning}, denoted by $\sum_{i \in [K]} B_i$ with $B_i \triangleq L_{ii} D_i^2$, measures the self-dependency within a single agent's domain.

\paragraph{New state-of-the-art communication cost.}
In the multiagent setting, Extragradient (\EG) remains the state-of-the-art method, requiring
\(
T ^{\EG} _{\sP_{\text{\upshape VIP}}} = \cO \bigl( \sum _{i \in [K]} \frac {A_i + B_i} {\epsilon} \bigr)
\)
communication cost, which depends on both the cross-coupled and diagonal conditionings.
By extending our decoupled template to distributed VIPs, we propose \DMvip, and establish a much better communication cost.

\begin{theorem}\label{thm:DMvip-communication-cost-short}
    Consider the \DMvip algorithm for problem class $\sP _{\text{\upshape VIP}}$. We have
    \[
    T ^{\DMvip} _{\sP_{\text{\upshape VIP}}} = \cO \Bigl( \sum _{i \in [K]} A_i / \epsilon \Bigr) .
    \]
\end{theorem}

\begin{remark}
    The communication cost in \cref{thm:DMvip-communication-cost-short} completely drops the dependence on the diagonal conditioning. Consequently, we have substantially improved the state-of-the-art \EG cost when the diagonal conditioning dominates, i.e., when $\sum _i B_i \gg \sum _i A_i$.
\end{remark}

\section{Conclusion and limitations}
\label{sec:conclusion}

This paper studies communication and oracle costs in distributed SPs and VIPs. For the class of SPs, we settle the communication complexity in the distributed setup within gradient-span framework, and consistently improve the long-standing oracle cost of \EG method. For the class of distributed VIPs, we improve the state-of-the-art communication cost. The following directions are not addressed in this paper and are left for future work: (a)~closing the gap of oracle costs; (b)~showing lower bound for non-zero-sum games; and (c)~showing information-theoretic lower bounds for randomized methods.

\bibliographystyle{plainnat} 
\bibliography{reference}

\clearpage
\appendix

\crefalias{section}{appendix} 

\section{Detailed review of existing algorithms}
\label{appendix:existing-algorithms}

In this section, we provide the detailed formulations, distributed trajectories, and formal complexity results for the algorithms summarized in \Cref{sec:existing-algorithms-for-SP}.
We will also show that all the algorithms described below indeed satisfy the gradient-span assumptions in \cref{def:gradient-span}.

We recall that $\sP_{\text{\upshape SP}}^\circ$ denotes the subclass of $\sP_{\text{\upshape SP}}$ where the local components $\psi_x$ and $\psi_y$ are zero. In this setting, the problem reduces to finding a saddle point of a smooth convex-concave function $f$. We note that many algorithms discussed below do not handle general composite functions and only deal with problem instances from this non-composite subclass.

As in \cref{remark:robustness-to-inexact-distance-estimates}, let us consider a practical scenario where the algorithms may not have the precise values of $D_x$ and $D_y$ in advance, but they have access to upper estimates $\hat D_x \ge D_x$ and $\hat D_y \ge D_y$.
Let
\[
    \boxed{
        \theta \triangleq \frac{D_x \hat D_y} {\hat D_x D_y} + \frac{D_y \hat D_x} {\hat D_y D_x},
    }
\]
which quantifies the disproportionality between the true distance parameters and their estimates.

\paragraph{Extragradient (\EG).}

Suppose we apply \EG to a non-composite problem instance $P = (f, 0, 0, \bfz^0) \in \sP ^\circ_{\text{\upshape SP}}$. The method iteratively maintains and updates two sequences of variables $\bfv^k, \bfz^k \in \cE_x \times \cE_y$. Initializing with $\bfv^0 = \bfz^0$, the updates for iteration $k = 0, 1, \dots, K-1$ are given by:
\begin{align*}
    \bfz^{k + 1} &= \bfv^k - \eta^k \bfP^{-1} V^f(\bfv^k) , \\
    \bfv^{k + 1} &= \bfv^k - \eta^k \bfP^{-1} V^f(\bfz^{k + 1}) ,
\end{align*}
where $\eta^k > 0$ is the step size at iteration $k$.

Although the \EG method is conventionally formulated as a $K$-iteration loop, its distributed execution requires $2K$ communication rounds. In each iteration, the agents must synchronize twice to evaluate the coupled partial gradients at $\bfv^k$ and at $\bfz^{k+1}$. Thus, the \EG method corresponds to a distributed gradient-span algorithm with $T = 2K$ rounds and local step lengths $\tau_x^t = \tau_y^t = 1$, where the trajectories are given by:
\[
\hat Z_x^{2k} = \hat Z_y^{2k} = \{ \bfv^k \}
\quad \text{and} \quad
\hat Z_x^{2k+1} = \hat Z_y^{2k+1} = \{ \bfz^{k+1} \} ,
\]
for $k = 0, \dots, K-1$.

To align with \Cref{def:distributed-gradient-based}, the method produces a candidate solution $\bar \bfz^{t+1}$ after each round $t \in \{0, \dots, T-1\}$. Because each iteration requires two communication rounds, the algorithm effectively updates its output only every two rounds. Specifically, after completing iteration $k$ (which corresponds to round $t = 2k+1$), it outputs the ergodic average $\bar \bfz^{2k+2} = \frac{1}{\sum_{i=0}^k \eta^i} \bigl( \sum_{i=0}^k \eta^i \bfz^{i+1} \bigr)$. During the intermediate rounds (at $t = 2k$), it simply retains the previous solution by setting $\bar \bfz^{2k+1} = \bar \bfz^{2k}$ (where $\bar \bfz^0 = \bfz^0$). Since the trajectories and the candidate solutions are formed by linear combinations of the evaluated gradients, they satisfy the span conditions in \cref{def:gradient-span}.

The complexity of the \EG method provides a natural baseline. Translating the classic results into our distributed complexity measures yields the following upper bounds.

\begin{proposition}[{\citealt[Eq.~(6.21)]{juditsky2011first}}]
\label[proposition]{proposition:EG-upper-bound-SP}
    Consider the \EG method applied to $\sP^\circ _{\text{\upshape SP}}$. With the parameter choices of $\alpha_x = \frac{L_x \hat D_x + L_{xy} \hat D_y}{\hat D_x}$, $\alpha_y = \frac{L_y \hat D_y + L_{xy} \hat D_x}{\hat D_y}$, and $\eta^k \equiv 1$, we have
    \[
    \begin{gathered}
    T^{\EG} _{{\sP}^\circ _{\text{\upshape SP}}} \le \theta \cdot \frac{L_{xy} D_x D_y} {\epsilon} + \frac {L_x D_x^2} {\epsilon} + \frac {L_y D_y^2} {\epsilon} , \\
    N^{\EG} _{{\sP}^\circ _{\text{\upshape SP}}} \le (c_x + c_y) \cdot \Bigl( \theta \cdot \frac{L_{xy} D_x D_y} {\epsilon} + \frac {L_x D_x^2} {\epsilon} + \frac {L_y D_y^2} {\epsilon} \Bigr) .
    \end{gathered}
    \]
\end{proposition}

\paragraph{Decoupled GDA.}

Let us apply DGDA to an instance $P \in \sP_{\text{\upshape SP}}^\circ$ with a fixed local trajectory length $\tau$. The algorithm maintains local iterates $\bfx^{t,l}$ and $\bfy^{t,l}$ for round $t = 0, \dots, T-1$ and local step $l = 0, \dots, \tau$. At the beginning of round $t$, the agents synchronize by exchanging their latest local iterates. Specifically, Agent~$x$ receives $\hat \bfy^t$ and Agent~$y$ receives $\hat \bfx^t$, defined as:
\[
\hat \bfx^t \triangleq \begin{cases} \bfx^0 & \text{if } t = 0 \\ \bfx^{t-1, \tau} & \text{if } t > 0 \end{cases}
\qquad \text{and} \qquad
\hat \bfy^t \triangleq \begin{cases} \bfy^0 & \text{if } t = 0 \\ \bfy^{t-1, \tau} & \text{if } t > 0 \end{cases} .
\]
The agents then initialize their local variables for the current round as $\bfx^{t,0} = \hat\bfx^t$ and $\bfy^{t,0} = \hat\bfy^t$. With the remote variables firmly fixed, the agents execute $\tau$ local gradient steps. For $l = 0, \dots, \tau - 1$, the local updates are given by:
\begin{align*}
    \bfx^{t,l+1} &= \bfx^{t,l} - \eta_x \bfP_x^{-1} \nabla_x f( \bfx^{t,l}, \hat\bfy^t ) , \\
    \bfy^{t,l+1} &= \bfy^{t,l} + \eta_y \bfP_y^{-1} \nabla_y f( \hat\bfx^t, \bfy^{t,l} ) ,
\end{align*}
where $\eta_x, \eta_y > 0$ are the local step sizes.

The DGDA algorithm yields the following trajectories:
\[
\hat Z_x^t = \bigl\{ ( \bfx^{t,l}, \hat\bfy^t ) \bigr\}_{l=0}^{\tau-1}
\quad \text{and} \quad
\hat Z_y^t = \bigl\{ ( \hat\bfx^t, \bfy^{t,l} ) \bigr\}_{l=0}^{\tau-1} .
\]
After each round $t \in \{0, \dots, T-1\}$, the method returns the updated local iterates as the candidate solution $\bar\bfz^{t+1} = ( \bfx^{t, \tau}, \bfy^{t, \tau} )$. Because the trajectories and the candidate solutions are formed entirely by linear combinations of the evaluated gradients, DGDA is indeed a distributed gradient-span algorithm by \cref{def:gradient-span}.

While DGDA lies perfectly in our framework, its theoretical guarantees are highly restrictive. The algorithm is only proven to converge for restricted strongly convex-strongly concave problem instances where the cross-coupling between the variables is sufficiently weak~\citep{zindari2025decoupled}. In this narrowly defined regime, DGDA achieves a logarithmic communication complexity of $\cO \bigl( \log \frac{1}{\epsilon} \bigr)$, which is a clear improvement over the \EG baseline. However, for general problem instances in $\sP_{\text{\upshape SP}}^\circ$ with stronger coupling, the delayed remote variables cause the local updates to drift, ultimately leading the method to diverge. Consequently, DGDA fails to provide any meaningful complexity guarantee for the general problem class $\sP _{\text{\upshape SP}}^\circ$ under consideration.

\paragraph{Catalyst acceleration.}

Catalyst methods first add small regularization terms to the objective:
$f (\bfx, \bfy) + \frac{\epsilon}{4 \hat D_x^2} \lVert \bfx - \bfx^0 \rVert_x^2 - \frac{\epsilon}{4 \hat D_y^2} \lVert \bfy - \bfy^0 \rVert_y^2$, reducing the problem to a strongly convex-strongly concave one.
Then, Catalyst introduces an outer loop, indexed by $k = 0, 1, \dots, K-1$, designed to balance the conditioning between the two variables. For instance, when the conditioning of $x$ is worse (i.e. $L_x \hat D_x^2 \ge L_y \hat D_y^2$), the method carefully maintains an extrapolation sequence $(\tilde \bfx^k) _{k=0}^{K-1}$ and, in each outer iteration, adds a proximal term $\frac{\lambda_x}{2} \lVert \bfx - \tilde{\bfx}^k \rVert_x^2$ to the objective. Conversely, if the conditioning of $y$ is worse, the outer loop would instead maintain an extrapolation sequence for $y$ and add a corresponding regularization term for $y$. An inner base algorithm (\EG in this case) is then deployed to solve this regularized subproblem to a specified accuracy.

To simplify the notation, let $L_{\max} = \max \{L_x, L_y, L_{xy}\}$.

\begin{proposition}[{\cite{yang2020catalyst,lan2026novel}}]
\label[proposition]{proposition:catalyst-upper-bound-SP}
    Consider the Catalyst framework equipped with the \EG method as the inner solver, denoted by Cat-EG, applied to $\sP_{\text{\upshape SP}}^\circ$. Then, we have:
    \[
    \begin{gathered}
    T^{\text{\upshape Cat-EG}} _{\sP_{\text{\upshape SP}}^\circ} = \cO \biggl( \biggl( \frac{L_{\max} \hat D_x \hat D_y}{\epsilon} + \sqrt { \frac {L_x \hat D_x^2 } {\epsilon} } + \sqrt { \frac {L_y \hat D_y^2} {\epsilon} } \biggr) \log^2 \Bigl( \frac{1}{\epsilon} \Bigr) \biggr), \\
    N^{\text{\upshape Cat-EG}} _{\sP_{\text{\upshape SP}}^\circ} = (c_x + c_y) \cdot \cO \biggl( \biggl( \frac{L_{\max} \hat D_x \hat D_y}{\epsilon} + \sqrt { \frac {L_x \hat D_x^2 } {\epsilon} } + \sqrt { \frac {L_y \hat D_y^2} {\epsilon} } \biggr) \log^2 \Bigl( \frac{1}{\epsilon} \Bigr) \biggr) .
    \end{gathered}
    \]
\end{proposition}

Now, let us explain the caveats we mentioned in \cref{sec:existing-algorithms-for-SP} regarding the Cat-EG method in more detail.

To explain the second caveat of Cat-EG, its sensitivity to inexact diameter estimates, we compare it with the \EG baseline and our proposed method. Because Catalyst uses $\hat D_x$ and $\hat D_y$ to set the initial regularization, its complexity scales directly with these estimates rather than the true distances $D_x$ and $D_y$. Specifically, the diagonal terms in \EG or \DMsp depend strictly on the true distances, whereas in Cat-EG they scale with $\hat D_x$ and $\hat D_y$. Similarly, for the cross-coupled term, \EG or \DMsp depends on the true distances multiplied by the proportionality ratio $\theta$. If the estimates are loose but proportional (e.g., $\hat D_x = c D_x$ and $\hat D_y = c D_y$), $\theta$ remains $2$, leaving the complexity unaffected by the overestimation factor $c$. In contrast, the Cat-EG coupled term scales with $\hat D_x \hat D_y$, meaning any overestimation of $\hat D_x \gg D_x$ or $\hat D_y \gg D_y$ directly inflates the bound. Consequently, \EG and our proposed \DMsp method are much more robust to inexact distance estimates, provided the estimates are roughly proportional.

Furthermore, we note that the fourth caveat of Cat-EG (performing worse than unaccelerated \EG) can be easily verified. For instance, consider a problem instance where $L_x = 10^{6}$, $L_{xy} = 1$, $L_y = 10^{-6}$, $D_x = 10^{-3}$, and $D_y = 10^{3}$. Under this conditioning, the theoretical upper bound of Cat-EG significantly exceeds that of standard \EG.

\paragraph{Four-loop method.}

The Cat-Cat-DAGDA method~\citep{wang2020improved} first adds small $\cO(\epsilon)$ regularizations to both $\bfx$ and $\bfy$ to ensure the objective is strongly convex-strongly concave. It then executes four nested loops, which justifies our naming convention:
(i)~The first loop is a Catalyst outer loop adding a proximal regularization term to the primal variable;
(ii)~The second loop is another Catalyst outer loop adding a proximal regularization term to the dual variable;
(iii)~The third loop manages communication by exchanging and freezing the remote variables, identical to DGDA; and
(iv)~The fourth loop performs local computations, but unlike the standard gradient steps in DGDA, it employs an accelerated gradient method (hence DAGDA) to solve the inner subproblems.

While the fourth loop of the method is a local computation loop, the first three loops all require communication rounds. Let us state the complexity results of this method without going into the tedious details of the algorithm.

\begin{proposition}[{\cite{wang2020improved}}]
\label[proposition]{proposition:PBR-upper-bound-SP}
    Consider the Cat-Cat-DAGDA method applied to $\sP_{\text{\upshape SP}}^\circ$. Then, we have:
    \[
    \begin{gathered}
    T^{\text{\upshape Cat-Cat-DAGDA}} _{\sP_{\text{\upshape SP}}^\circ} = \cO \biggl( \frac{{L_{xy}} \hat D_x \hat D_y}{\epsilon} \log^3 \Bigl( \frac{1}{\epsilon} \Bigr) \biggr) , \\
    N^{\text{\upshape Cat-Cat-DAGDA}} _{\sP_{\text{\upshape SP}}^\circ} = (c_x + c_y) \cdot \cO \biggl( \biggl( \frac{\sqrt{L_{\max} L_{xy}} \hat D_x \hat D_y}{\epsilon} + \sqrt { \frac {L_x \hat D_x^2 } {\epsilon} } + \sqrt { \frac {L_y \hat D_y^2} {\epsilon} } \biggr) \log^4 \Bigl( \frac{1}{\epsilon} \Bigr) \biggr) .
    \end{gathered}
    \]
\end{proposition}

\section{MRN solver: Accumulative Regularization Method}
\label{sec:MRN-solver-ARM}

\begin{algorithm}[H]
\caption{$\ARM \bigl( \nabla f_w, \psi_w, \bfv_w, \xi \mid L \bigr)$}\label{alg:accumulative_regularization}
\begin{algorithmic}[1]
\State Set $\tau = 2 + \max \bigl\{ 0, \bigl\lceil \log_4 \bigl( \frac{3 L}{2 \xi} \bigr) \bigr\rceil \bigr\}$ and $\sigma^{(0)} = 0$.
\State Set $\sigma^{(k)} = 4^{k-3} \frac {2\xi} {3}$ and $N_k = \Bigl\lceil 16 \sqrt{\frac{ L}{\sigma^{(k)}}} \Bigr\rceil$ for $k = 1, \dots, \tau$.
\State Initialize $\bar{\bfw}^{(0)} = \bfw^{(0)} = \bfv_w$.
\For{$k = 1, \dots, \tau$}
    \State $\gamma^{(k)} = 1 - \frac{\sigma^{(k-1)}}{\sigma^{(k)}}$
    \State $\bar{\bfw}^{(k)} = (1 - \gamma^{(k)})\bar{\bfw}^{(k-1)} + \gamma^{(k)} \bfw^{(k-1)}$

    \vspace{0.1cm}
    \Statex \hspace{0.45cm} \emph{\% Begin Inner Subroutine: Nesterov's Accelerated Gradient Method}
    \State Initialize $\bfx_k^{(0)} = \bfw^{(k-1)}$, $\bfy_k^{(0)} = \bfw^{(k-1)}$, $t_0 = 1$, and $L_k = L + \sigma^{(k)}$.
    \For{$i = 0, \dots, N_k - 1$}
        \State $\nabla f _w ^{(k)} (\bfy_k^{(i)}) = \nabla f _w (\bfy_k^{(i)}) + \sigma^{(k)} (\bfy_k^{(i)} - \bar{\bfw}^{(k)})$
        \State \label{line:xx-update}$\bfx_k^{(i+1)} = \argmin_{\bfw \in \dom \psi_w}\ \bigl\{ \InnerProduct {\nabla f_w ^{(k)}(\bfy_k^{(i)})} {\bfw} + \frac{L_k}{2} \Norm{\bfw - \bfy_k^{(i)}}_w^2 + \psi_w (\bfw) \bigr\}$
        \State $t_{i+1} = \frac{1 + \sqrt{1 + 4t_i^2}}{2}$
        \State $\bfy_k^{(i+1)} = \bfx_k^{(i+1)} + \frac{t_i - 1}{t_{i+1}} (\bfx_k^{(i+1)} - \bfx_k^{(i)})$
    \EndFor
    \Statex \hspace{0.45cm} \emph{\% End Inner Subroutine}
    \vspace{0.1cm}

    \State $\bfw^{(k)} = \bfx_k^{(N_k)}$
    \State $\psi _w ^\prime(\bfw^{(k)}) = - \nabla f _w ^{(k)} (\bfy_k^{(N_k - 1)}) - L_k (\bfx_k^{(N_k)} - \bfy_k^{(N_k - 1)})$
\EndFor
\State \Return $(\bfw^{(\tau)}, \psi _w ^\prime(\bfw^{(\tau)}))$
\end{algorithmic}
\end{algorithm}

\begin{proof}[Proof of \cref{lemma:making-the-gradient-norm-small-full}]
    It was originally shown in \cite[Theorem~3.1]{lan2023optimal} that $\ARM$ takes no more than $34 \sqrt { \frac {3 L} {2 \xi} } $ gradient queries and obtains $\bfw^\prime \in \dom \psi _w$, such that there exists $\tilde \bfw \in \dom \psi _w$ with $\nabla f _w (\tilde \bfw) \in - \partial \psi _w (\tilde \bfw)$, and
    \[
    \Norm{ 2L (\bfw^{+} - \bfw^\prime) } _{w} \le \frac {2} {3} \xi \Norm {\bfv_w - \tilde \bfw} _{w}, \text{ where }
    \bfw^{+} = \argmin _{\bfw \in \dom \psi _w} \ \bigl[ \InnerProduct { \nabla f _w (\bfw^\prime)} {\bfw } + \psi _w (\bfw) + L \lVert \bfw - \bfw^\prime \rVert _{w} ^2 \bigr] .
    \]
    By the optimality of $\bfw^{+}$, there exists $\psi^\prime _w (\bfw^+) \in \partial \psi _w (\bfw^+)$ such that
    \[
    \nabla f _w (\bfw^\prime) + \psi^\prime _w (\bfw^{+}) + 2 L \bfP (\bfw^{+} - \bfw^\prime) = \0 .
    \]
    Then, we have
    \[
    \begin{aligned}
    &\quad \Norm {\nabla f _w (\bfw^+) + \psi _w ^\prime(\bfw^+) } _{w^*} \\
    &\le \Norm { \nabla f _w (\bfw^\prime) + \psi^\prime _w (\bfw^{+}) } _{w^*} + \Norm { \nabla f _w (\bfw^+) - \nabla f _w(\bfw^\prime)} _{w^*} \\
    &= \Norm {2 L \bfP (\bfw^{+} - \bfw^\prime)} _{w^*} + \Norm { \nabla f_w (\bfw^+) - \nabla f_w (\bfw^\prime)} _{w^*} \\
    &\le 3 L \Norm { \bfw^+ - \bfw^\prime } _{w} \\
    &\le \xi \Norm {\bfv_w - \tilde \bfw} _{w} .
    \end{aligned}
    \]
\end{proof}

\section{Monotone composite variational inequality problems}
\label{sec:monotone-block-VIPs-appendix}

In this section, we study variational inequality problems~(VIPs)~\citep{nemirovski2004prox,juditsky2011first}, a generalization of SPs that captures, for instance, multiplayer general-sum games.

\subsection{Problem formulation}

\paragraph{VIPs~(with separable composite terms).}
Let us consider the VIP in \cref{eq:composite-VIP}, where $\cE = \cE_1 \times \cdots \times \cE_K$ is the direct product of $K$ finite-dimensional real vector spaces.
For all $i \in [K]$: let the mapping $V_i \colon \dom \psi \to \cE_i^*$, and let the function $\psi_i \colon \cE_i \to \R \cup \{ + \infty \}$.
We consider the decomposition of $V(\bfz) = \bigl( V_1(\bfz), \cdots, V_K (\bfz) \bigr)$ and $\psi (\bfz) = \psi_1 (\bfz_1) + \cdots + \psi_K (\bfz_K)$, for all $\bfz = (\bfz_1, \cdots, \bfz_K) \in \cE$.
Moreover, we denote $\dom \psi = \dom \psi_1 \times \cdots \times \dom \psi_K \triangleq Q$.

\paragraph{Assumptions for VIPs.} Let us make the following assumptions:
\begin{enumerate}[label={\textnormal{(A\arabic*')}}]
    \setcounter{enumi}{\value{VIP-assumptions}}
    \item  \label{item:bounded-distance-to-block-solution}Let $\bfz^0 = (\bfz^0_1, \cdots, \bfz^0_K) \in Q$ be a given point. There exists $\bfz^* = (\bfz_1^*, \cdots, \bfz_K^*) \in Q$ in the solution set of the VIP of $(V, \psi)$, such that for all $i \in [K]$: $\bfz_i^* \in \cB_i$, where $\cB_i \triangleq \{ \bfz_i \in \cE_i \mid \Norm {\bfz_i^0 - \bfz_i}_i \le D_i \}$ and $D_i > 0$ is a given distance.
    \item \label{item:operator-block-Lipschitz-continuity}The operator $V_i (\bfz_j; \bfz_{-j})$ is $L_{ij}$-Lipschitz continuous in $\bfz_j \in \dom \psi_j$ for any fixed $\bfz_{-j} \in \dom \psi_1 \times \cdots \times \dom \psi_{j-1} \times \dom \psi_{j+1} \times \cdots \times \dom \psi_K$.
    \footnote{For all $\bfz = (\bfz_1, \cdots, \bfz_K) \in Q$, we use the following notations for simplicity: $(\bfz_j; \bfz_{-j}) \triangleq \bfz$ and $\bfz_{-j} \triangleq (\bfz_1, \cdots, \bfz_{j-1}, \bfz_{j+1}, \cdots, \bfz_K )$.}
    \setcounter{Block-VIP-assumptions}{\value{enumi}}
\end{enumerate}
Let the operator family $\sF _{\text{\upshape VIP}}$ be comprised of all the operators with initialization points $\bigl( (V_i) _{i \in [K]}, \bfz^0 \bigr) \in \sP _{\text{\upshape VIP}}$, such that Assumptions~\labelcref{item:a-monotone-operator-and-a-convex-function,item:operator-Lipschitz-continuity,item:bounded-distance-to-block-solution} are satisfied.

\paragraph{Notations.}
To simplify the notations, let us denote $\bfz \triangleq (\bfz_1, \cdots, \bfz_K) \in Q$ in the context of VIPs. Let us denote
\[
\bar L_{ij} \triangleq \max \{ L_{ij}, L_{ji} \}, \
A_i \triangleq D_i \Bigl( \sum _{ j \in [K] \setminus \{ i \} } \bar L_{ij} D_j \Bigr), \text{ and } B_i \triangleq \bar L_{ii} D_i^2, \text{ for all } i, j \in [K].
\]
We refer to $\sum _i A_i$ as the cross-coupled conditioning and $\sum _i B_i$ as the diagonal conditioning, and we say that the diagonal conditioning dominates when $\sum _i B_i \gg \sum _i A_i$.

\subsection{Communication and computational costs}

\paragraph{Distributed oracles.}

We consider a distributed setting with $K$ agents very similar to the one in \cref{sec:problem-class,sec:distributed-gradient-based-algorithms}. For all $i \in [K]$: Agent~$i$ controls decision variable $\bfz_i \in \dom \psi_i$, has direct access to the function $\psi_i$, and has access to the oracle $\sO _i (\bfz) = V_i (\bfz)$ for $\bfz \in Q$.
We consider per query to $\sO_i$ costs $c_i \ge 0$, $i \in [K]$.

\paragraph{Accuracy measure.}
We consider the following accuracy measure for VIPs:
\[
\Delta (\bar \bfz) \triangleq \sup _{ \bfz \in \cB \cap Q} \InnerProduct { V (\bfz) }{ \bar \bfz - \bfz } + \psi (\bar \bfz) - \psi (\bfz) , \text{ for all } \bfz \in Q ,
\]
where $\cB \triangleq \cB_1 \times \cdots \times \cB_K$.
We say that a point $\bar \bfz \in Q$ is an \emph{$\epsilon$-approximate solution} of the VIP if
\(
\Delta (\bar \bfz) \le \epsilon .
\)
Our goal is to find such an $\epsilon$-approximate solution for any $\epsilon > 0$.

\paragraph{Problem class.}
We formally define the overall \emph{problem class}, denoted by $\sP_{\text{\upshape VIP}} \bigl( \sF _{\text{\upshape VIP}}, (\sO_i) _{i \in [K]}, \epsilon \bigr)$, or for short $\sP_{\text{\upshape VIP}}$. A specific problem instance $P \in \sP_{\text{\upshape VIP}}$ is constructed by drawing an operator instance (with initial point) $V$ from $\sF _{\text{\upshape VIP}}$, equipping it with the distributed oracles $(\sO_i) _{i \in [K]}$, and specifying a target accuracy $\epsilon > 0$. Solving the instance $P$ requires an algorithm to output an $\epsilon$-approximate solution of $V$ utilizing the distributed oracles.

\paragraph{Distributed algorithms for VIPs, communication and oracle costs.}

To provide the formal definitions of distributed algorithms for the problem class $\sP_{\text{\upshape VIP}}$, we generalize the information-based framework to $K$ agents. Every query point, message, and output is generated as a deterministic mapping of the information available to the agent at that step.

Suppose an algorithm $\sM$ proceeds in $T$ rounds. In each round $t \in \{0, \dots, T-1\}$, each Agent~$i \in [K]$ executes multiple local computational steps to generate local query points denoted by
\[
\bfz_i^{t,l} = (\bfz_{i,1}^{t, l}, \dots, \bfz_{i,i}^{t, l}, \dots, \bfz_{i,K}^{t, l}) \quad \text{for } l \in \{0, \dots, \tau_i^t - 1\} ,
\]
where $\bfz_{i,i}^{t,l}$ is the local variable updated by Agent~$i$, and $\bfz_{i,j}^{t, l}$ (for $j \neq i$) is the delayed approximation of Agent~$j$'s variable utilized by Agent~$i$.

Let $I_i^{t,l}$ denote the accumulated information sequence available to Agent~$i$ prior to making its $(l+1)$-th local oracle query in round $t$. The base case at initialization is $I_i^{0,0} = (\psi_i, \bfz^0)$. During the local computational steps $l \in \{0, \dots, \tau_i^t - 1\}$, the information sequence of Agent~$i$ updates sequentially by appending the newly acquired oracle response for $V_i$:
\[
I_i^{t, l+1} = \bigl( I_i^{t, l}, V_i(\bfz_i^{t, l}) \bigr) .
\]
After the local steps in round $t$, the agents exchange messages. Let $M_{j \to i}^t$ denote the message sent from Agent~$j$ to Agent~$i$. The information sequence available to Agent~$i$ at the beginning of round $t+1$ appends the received messages to its prior local history:
\[
I_i^{t+1, 0} = \bigl( I_i^{t, \tau_i^{t}}, \{ M_{j \to i}^t \}_{j \neq i} \bigr) .
\]

\begin{definition}
\label{def:distributed-vip-algorithm}
An algorithm $\sM$ is called a \emph{distributed algorithm} for problem class $\sP_{\text{\upshape VIP}}$ if, when applied to any instance $P \in \sP_{\text{\upshape VIP}}$, its execution satisfies the following conditions for all $t \in \{0, \dots, T-1\}$ and agents $i \in [K]$:
\begin{enumerate}
    \item \textbf{Local Computation:} The query points are determined entirely by the locally available information. For all local steps $l \in \{0, \dots, \tau_i^t - 1\}$, there exists a deterministic mapping function $\mathcal{A}_i^{t,l}$ such that:
    \[
        \bfz_i^{t,l} = \mathcal{A}_i^{t,l} \bigl( I_i^{t,l} \bigr) .
    \]

    \item \textbf{Communication:} The messages exchanged are produced by deterministic mappings of the sender's local information. For any $j \neq i$, there exists a mapping function $\mathcal{M}_{i \to j}^t$ such that:
    \[
        M_{i \to j}^t = \mathcal{M}_{i \to j}^t\bigl( I_i^{t, \tau_i^t} \bigr) .
    \]

    \item \textbf{Candidate solution:} The candidate solutions $\bar \bfz^{t+1} = (\bar \bfz_1^{t+1}, \dots, \bar \bfz_K^{t+1})$ are constructed from the respective agents' updated information sets. There exists a deterministic mapping function $\bar{\mathcal{A}}_i^{t+1}$ such that:
    \[
        \bar \bfz_i^{t+1} = \bar{\mathcal{A}}_i^{t+1} \bigl( I_i^{t+1, 0} \bigr) .
    \]
\end{enumerate}
\end{definition}

For a given instance $P \in \sP_{\text{\upshape VIP}}$ and a target accuracy $\epsilon > 0$, we define the \emph{communication cost} required by a distributed algorithm $\sM$ on $P$, denoted by $T^\sM_P$, as the smallest integer $k \in \{ 1, \dots, T \}$ such that the candidate solution $\bar \bfz^k$ satisfies the target accuracy $\epsilon$.

The total number of local oracle queries evaluated by Agent~$i$ for instance $P$ up to this point is given by the cumulative number of local steps taken, denoted by $N_{i,P}^\sM = \sum_{r=0}^{T^\sM_P-1} \tau_i^r$.

Let $c_i$ denote the computational cost of evaluating a single partial oracle $V_i$. The \emph{communication cost} and \emph{oracle cost} of algorithm $\sM$ over the entire problem class $\sP_{\text{\upshape VIP}}$ are defined by taking the supremum over all instances:
\[
T^\sM_{\sP_{\text{\upshape VIP}}} = \sup_{P \in \sP_{\text{\upshape VIP}}} T^\sM_P
\quad \text{and} \quad
N_{\sP_{\text{\upshape VIP}}}^\sM = \sup_{P \in \sP_{\text{\upshape VIP}}}\ \Bigl( \sum_{i \in [K]} c_i N_{i,P}^\sM \Bigr) .
\]

We first state the classic results of the \EG method in \cref{proposition:EG-upper-bound-block-VIPs}, which remains the state-of-the-art communication complexity bound.
\begin{proposition}[{\citealt[Eq.~(6.21)]{juditsky2011first}}]
\label[proposition]{proposition:EG-upper-bound-block-VIPs}
    For any target accuracy $\epsilon > 0$, the communication cost of \EG is bounded by
    \[
    \cO \bigl( \sum _{i \in [K]} \frac {A_i + B_i} {\epsilon} \bigr) ,
    \]
    and the computational cost of \EG is bounded by
    \[
    \cO \Bigl( \bigl( \sum _{i \in [K]} c_i \bigr) \bigl( \sum _{i \in [K]} \frac {A_i} {\epsilon} \bigr) + \bigl( \sum _{i \in [K]} c_i \bigr) \bigl( \sum _{i \in [K]} \frac {B_i} {\epsilon} \bigr) \Bigr) .
    \]
\end{proposition}

\subsection{Decoupled method for variational inequality problems}
\label{sec:dm-for-vips}



Now, we present our \DMvip method, which extends the \DMsp into multiplayer general-sum games.

\paragraph{Assembled norm.}

Given parameters $\alpha_i > 0$ for all $i \in [K]$ (to be specified later), we equip the joint space $\cE = \cE_1 \times \dots \times \cE_K$ with the assembled norm:
\begin{equation}\label{eq:assembled-norm-VIP}
\Norm {\bfz} _{\cE} = {\InnerProduct {\bfP \bfz}{\bfz}} ^{\frac 1 2} = \sqrt { \sum_{i=1}^K \alpha_i \Norm {\bfz_i} _i ^2 } \quad \text{for all $\bfz \in \cE$,}
\end{equation}
which corresponds to the block diagonal linear operator \( \bfP = \alpha_1 \bfP_1 \oplus \dots \oplus \alpha_K \bfP_K \). Accordingly, we equip the dual space $\cE^* = \cE_1^* \times \dots \times \cE_K^*$ with the corresponding dual norm:
\begin{equation}\label{eq:assembled-dual-norm-VIP}
\Norm {\bfg} _{\cE^*} = {\InnerProduct{\bfg}{\bfP^{-1} \bfg}}^{\frac 1 2} = \sqrt{ \sum_{i=1}^K \alpha_i^{-1} \Norm {\bfg_i} _{i^*} ^2 } \quad \text{for all $\bfg \in \cE^*$.}
\end{equation}

\paragraph{Template \DMvip.}

To extend our decoupled framework to block composite variational inequality problems (VIPs), we first define the coupled conditioning constant, which characterizes the interaction between the $K$ distinct blocks:
\begin{equation}
\label{eq:coupled-conditioning}
\bar L_{\texttt{c}} \triangleq \sqrt {\max _{j \in [K]} \biggl[ (\alpha_j D_j) ^{-1} \sum _{i \in [K] \setminus \{ j \}} \frac { \bar L_{ij} \bigl( \sum_{l \in [K] \setminus \{i\}} \bar L_{il} D_l \bigr) } { \alpha_i } \biggr]} .
\end{equation}

\Cref{alg:template-decoupled-reduced-operator-method-for-block-VIP} outlines the Decoupled Method for block composite VIPs~(\DMvip), generalizing the $\DMsp$ procedure. The algorithm maintains a sequence of anchor points $\bfv^t$ and orchestrates iterative updates among $K$ agents over a distributed network.

At the start of each iteration, the agents decouple the joint problem by fixing their remote variables to the current anchor components $\bfv^t_{-i}$. This allows each Agent~$i$ to independently and concurrently solve its regularized local subproblem. Specifically, each agent invokes an internal solver $\cM^{\MRN}_i$ to minimize the local residual norm (MRN) up to a target accuracy $\delta_i^{t+1} = \frac{\alpha_i \lambda_{t+1}}{2}$~(Line~\ref{line:find-small-gradient-norm-block-VIP}).

Following this local computation phase, the $K$ agents perform exactly two communication rounds to complete the iteration. In the first round (Line~\ref{line:exchange-variables-vip}), the agents broadcast their locally computed approximate solutions $\bfz_i^{t+1}$ to assemble the joint intermediate point $\bfz^{t+1}$. In the second round (Line~\ref{line:compute-extragradient-vip}), they use this assembled point to evaluate their local partial operators $V_i(\bfz^{t+1})$, which they subsequently exchange to form the full joint operator $V_\psi(\bfz^{t+1})$.

Finally, using this assembled joint operator, the agents compute a closed-form reduced step size $a_{t+1}$, update the running ergodic average $\bar \bfz^{t+1}$, and perform a joint extragradient-like step to generate the next anchor $\bfv^{t+1}$~(Lines~\ref{line:compute-reduced-stepsize-vip} and~\ref{line:reduced-gradient-step-vip}).

\begin{algorithm}[htb]
\caption{$\DMvip \bigl(K, (V_i)_{i \in [K]}, (\psi_i)_{i \in [K]}, \bfz^0, (\lambda_{t})_{t \ge 1}, (\alpha_i)_{i \in [K]} \mid (\cM ^{\MRN} _i) _{i \in [K]}\bigr)$}
\label{alg:template-decoupled-reduced-operator-method-for-block-VIP}
\begin{algorithmic}[1]
    \Require A local solver $\cM ^{\MRN}_i$ for the minimization of residual norms.
    \State $\bfv^{0} = (\bfv^{0}_1, \dots, \bfv^{0}_K) = \bfz^0$.
    \For {$t=0,1,\dots,T-1$}
        \State Let $\delta^{t+1}_i = \frac {\alpha_i \lambda_{t+1}} {2}$ for all $i \in [K]$.
        \State \label{line:find-small-gradient-norm-block-VIP}Concurrently for all $i \in [K]$, Agent~$i$ computes
        \[
        \bigl(\bfz^{t+1}_i, \psi_i ^\prime (\bfz^{t+1}_i)\bigr)
        = \cM ^{\MRN} _i \Bigl(V_i (\cdot; \bfv^t_{-i}),\, \psi_i + \frac {\alpha_i \lambda_{t+1}} {2} \Norm {\cdot - \bfv^t_i}_i^2,\, \bfv^t_i,\, \delta^{t+1}_i \Bigr) .
        \]
        \State \label{line:exchange-variables-vip}All agents exchange $\bfz^{t+1}_i$ to assemble the joint point $\bfz^{t+1} = (\bfz^{t+1}_1, \dots, \bfz^{t+1}_{K})$.
        \State \label{line:compute-extragradient-vip}Agents locally compute $V_i(\bfz^{t+1})$ and exchange to assemble the full operator:
        \[
        V _\psi (\bfz^{t+1}) = V (\bfz^{t+1}) +  \bigl(\psi_1^\prime (\bfz^{t+1}_1), \dots, \psi_K^\prime (\bfz^{t+1}_K)\bigr) .
        \]
        \State \label{line:compute-reduced-stepsize-vip}Let $a_{t+1} = \frac { 2 \InnerProduct {V _\psi (\bfz^{t+1})}{ \bfv^t - \bfz^{t+1} } } { \Norm {V _\psi (\bfz^{t+1})} _{\cE^*} ^2 }$ and generate solution $\bar \bfz^{t+1} = \bigl( \sum _{i=1} ^{t+1} a_{i} \bigr) ^{-1} \sum _{i=1} ^{t+1} a_{i} \bfz^{i}$.
        \State \label{line:reduced-gradient-step-vip}$\bfv^{t+1} = \argmin _{\bfv \in Q}\ \bigl[ a_{t+1} \InnerProduct {V _\psi (\bfz^{t+1})}{\bfv} + \frac 1 2 \Norm {\bfv - \bfv^t} _{\cE} ^2 \bigr]$.
    \EndFor
\end{algorithmic}
\end{algorithm}

Similar to the saddle point setting, we refer to \cref{alg:template-decoupled-reduced-operator-method-for-block-VIP} as a template method because we have abstracted the exact implementation of the inner solvers. For the theoretical guarantees presented below, we merely assume the existence of an algorithmic oracle $\cM ^{\MRN} _i (\hat V_i, \hat \psi_i, \bfv_i, \delta_i)$ capable of taking an MRN instance in the space $\cE_i$ and successfully returning a sufficiently accurate solution. We defer the detailed implementation of these local solvers to \cref{eq:final-implementation-of-DROMbcvip} at the end of this section.
For now, let us proceed with the communication complexity of the template \DMvip method.

\begin{theorem}
\label[theorem]{thm:DROMbcvip-communication-cost}
    Consider the $\DMvip$ template applied to $\sP_{\text{\upshape VIP}}$. With the parameter choices of $\alpha_i = \frac { \sum _{j \in [K] \setminus \{i\} } \bar L_{ij} D_j } {D_i}$ for all $i \in [K]$, and $\lambda_{t} \equiv \lambda \ge 2 \bar L_{\texttt{c}}$, we have:
    \[
    T ^{\DMvip} _{\sP_{\text{\upshape VIP}}} \le 2 + \sum _{\substack{i, j \in [K] \\ i \neq j}} \frac {2 \bar L_{ij} D_i D_j} {\epsilon} .
    \]
\end{theorem}
For the distributed VIP considered in this paper, the classic \EG method represents the best known communication cost. Now, let us compare our communication guarantee with that of \EG.

\begin{remark}[Improved communication]
\label[remark]{remark:improved-communication-cost-VIP-diagonal-conditioning-dominates}
    The classic \EG method takes
    \[
    \frac {1} \epsilon \sum _{i \in [K]} (A_i + B_i)
    \]
    communication rounds~(cf.~\cref{proposition:EG-upper-bound-block-VIPs}).
    Our communication cost in \cref{thm:DROMbcvip-communication-cost} is $\cO \bigl( \sum _{i \in [K]} A_i / \epsilon \bigr)$, which is consistently no worse compared to that of \EG and is substantially faster when the ``diagonal conditioning'' dominates---i.e.,
    \[
    \sum _{i \in [K]} B_i \gg \sum _{i \in [K]} A_i .
    \]
    To our knowledge, \DMvip improves the state-of-the-art communication cost for distributed VIPs.
\end{remark}

\subsection{Detailed proofs}

\subsubsection{Proof for \FDS}

Let us provide the detailed pseudocode of \FDS for VIPs in \cref{alg:fully-decoupled-solver-VIP}.
Then, we prove the correctness of the solution returned by \FDS.
\setlength{\textfloatsep}{0.2cm}
\begin{algorithm}
\caption{$\FDS _{\Norm{\cdot} _{\cE}} \bigl( (V_i) _{i \in [K]}, (\psi_i) _{i \in [K]}, \bfv, \lambda \mid (\cM ^{\MRN} _i) _{i \in [K]} \bigr)$}
\label{alg:fully-decoupled-solver-VIP}
\begin{algorithmic}[1]
\Require Solver $\cM ^{\MRN} _i$ for the minimization of residual norms, for all $i \in [K]$.
    \For {$i \in [K]$}
        \State $\delta_i = \frac {\alpha_i \lambda} {2}$.
        \State $\hat \psi_i = \psi_i + \frac{\alpha_i \lambda}{2} \Norm {\cdot - \bfv_i} _i ^2$.
        \State $(\bfz_i^+, \psi^\prime _i (\bfz_i^+)) = \cM ^{\MRN} _i \bigl(V_i (\cdot; \bfv_{-i}), \hat \psi_i, \bfv_i, \delta_i \bigr)$.
        \State $\psi_i ^\prime (\bfz_i ^+) = \hat \psi _i ^\prime (\bfz_i^+) - \alpha_i \lambda \bfP_i (\bfz_i ^+ - \bfv_i)$.
    \EndFor

    \State \Return $(\bfz^+, \psi ^\prime (\bfz^+) )$, where $ \bfz^+ = (\bfz^+_1, \cdots, \bfz^+_K)$ and $\psi ^\prime (\bfz^+) = (\psi ^\prime _1 (\bfz_1^+), \cdots, \psi ^\prime _K (\bfz_K^+))$.
\end{algorithmic}
\end{algorithm}
\setlength{\floatsep}{0.2cm}

\begin{lemma}
\label[lemma]{lemma:FDS-correctness-VIP}
Under \labelcref{item:operator-block-Lipschitz-continuity}, for $\lambda \ge 2 \bar L_{\texttt{c}}$, \FDS~(\cref{alg:fully-decoupled-solver-VIP}) returns the correct solution of the MS subproblem given by $(V, \psi, \bfv, \lambda)$.
\end{lemma}

\begin{proof}[Proof of \cref{lemma:FDS-correctness-VIP}]
    For all $i \in [K]$, by \labelcref{item:operator-block-Lipschitz-continuity} and then by the relative distance accuracy, we have
    \begin{equation}
    \label{eq:FDS-bound-on-block-coordinate-i}
    \begin{aligned}
    & \Norm { V_i (\bfz^{+}) + \psi_i^\prime (\bfz^{+}_i) + \alpha_i \lambda \bfP_i (\bfz^{+}_i - \bfv_i)} _{i^*} \\
    &\quad \le
    \Norm { V_i (\bfz^{+}_i; \bfv_{-i}) + \psi_i^\prime (\bfz^{+}) + \alpha_i \lambda \bfP_i (\bfz^{+}_i - \bfv_i)} _{i^*} + \sum _{j \in [K] \setminus \{i\}} L_{ij} \Norm {\bfz^+_j - \bfv_j} _j \\
    &\quad \le \delta_i \Norm {\bfz^+_i - \bfv_i} _{i} + \sum _{j \in [K] \setminus \{i\}} L_{ij} \Norm {\bfz^+_j - \bfv_j} _j .
    \end{aligned}
    \end{equation}

    Finally, we assemble the norms:
    \[
    \begin{aligned}
    & \Norm { V (\bfz^{+}) + \psi^\prime (\bfz^{+}) + \lambda \bfP (\bfz^{+} - \bfv)} _{\cE^*} ^2 \\
    &\quad = \sum _{i \in [K]} \alpha_i^{-1} \Norm { V_i (\bfz^{+}) + \psi_i^\prime (\bfz^{+}_i) + \alpha_i \lambda \bfP_i (\bfz^{+}_i - \bfv_i)} _{i^*} ^2 \\
    &\quad \refLE{eq:FDS-bound-on-block-coordinate-i} \sum _{i \in [K]} \alpha_i^{-1} \Bigl( \delta_i \Norm {\bfz^+_i - \bfv_i} _{i} + \sum _{j \in [K] \setminus \{i\}} L_{ij} \Norm {\bfz^+_j - \bfv_j} _j \Bigr) ^2 \\
    &\quad \le \sum _{i \in [K]} 2 \alpha_i^{-1} \biggl[ \delta_i^2 \Norm {\bfz^+_i - \bfv_i} _{i} ^2 + \Bigl( \sum _{j \in [K] \setminus \{i\}} L_{ij} \Norm {\bfz^+_j - \bfv_j} _j \Bigr) ^2 \biggr] \\
    &\quad = \frac {\lambda^2} {2} \sum _{i \in [K]} \Bigl( \alpha_i \Norm {\bfz^+_i - \bfv_i} _{i} ^2 \Bigr) + 2 \sum _{i \in [K]} \biggl[ \alpha_i^{-1} \Bigl( \sum _{j \in [K] \setminus \{i\}} L_{ij} \Norm {\bfz^+_j - \bfv_j} _j \Bigr) ^2 \biggr] \\
    &\quad \le \frac {\lambda^2} {2} \Norm {\bfz^+ - \bfv} _{\cE} ^2 + 2 \sum _{i \in [K]} \biggl[ \alpha_i^{-1} \Bigl( \sum _{l \in [K] \setminus \{i\}} L_{il} D_l \Bigr) \Bigl( \sum _{j \in [K] \setminus \{i\}} \frac {L_{ij}} {D_j} \Norm {\bfz^+_j - \bfv_j} _j ^2 \Bigr) \biggr] \\
    &\quad = \frac {\lambda^2} {2} \Norm {\bfz^+ - \bfv} _{\cE} ^2 + 2 \sum _{j \in [K]} \biggl[ \frac { \lVert {\bfz^+_j - \bfv_j} \rVert _j ^2 } { D_j } \Bigl( \sum _{ i \in [K] \setminus \{ j \} } \frac { L_{ij} \bigl( \sum _{l \in [K] \setminus \{i\}} L_{il} D_l \bigr) } { \alpha_i } \Bigr) \biggr] \\
    &\quad \refLE{eq:coupled-conditioning} \frac {\lambda^2} {2} \Norm {\bfz^+ - \bfv} _{\cE} ^2 + 2 \Norm {\bfz^+ - \bfv} _{\cE} ^2 \bar L_{\texttt{c}}^2 \\
    &\quad \le \frac {\lambda^2} {2} \Norm {\bfz^+ - \bfv} _{\cE} ^2 + \frac {\lambda^2} {2} \Norm {\bfz^+ - \bfv} _{\cE} ^2 = {\lambda^2} \Norm {\bfz^+ - \bfv} _{\cE} ^2 .
    \end{aligned}
    \]
\end{proof}

\subsubsection{Proof for \MRN}

Our algorithm is built upon
\cref{lemma:making-the-residual-norm-small-full}, the proof of which can be found in \cite[Corollary~2.4]{boct2024extra}.

\begin{lemma}
\label[lemma]{lemma:making-the-residual-norm-small-full}
    Assume \labelcref{item:a-monotone-operator-and-a-convex-function}, \labelcref{item:operator-Lipschitz-continuity}, and that the solution set of the VIP of $(V _w, \psi _w)$ is non-empty.
    Then, there exists an algorithm, denoted by
    \(
    ( \bfw^+, \psi _w ^\prime (\bfw^+) ) = \FEGM\ (V _w, \psi _w, \bfv_{w}, \xi \mid L ),
    \)
    which takes no more than
    \(
    C_0 \cdot { \frac {L} {\xi} } 
    \)
    operator queries and returns $( \bfw^+, \psi _w ^\prime (\bfw^+) )$ that satisfies $\xi$-distance-to-solution accuracy, where $C_0 > 0$ is some fixed constant.
\end{lemma}

\subsubsection{Concrete implementation}

We are now back to considering the VIPs.
Let us use $\FEGM$ in \cref{lemma:making-the-residual-norm-small-full} for the minimization of residual norms:
\[
\cM ^{\FEGM} _{i} (\hat V_i, \hat \psi_i, \bfv_i, \delta_i ) \triangleq \FEGM (\hat V_i, \hat \psi_i, \bfv_i, \frac {2 \delta_i} {3} \mid L_{ii} ) .
\]
Then, for any Monteiro-Svaiter Subproblem given by $(V, \psi, \bfv, \lambda)$, we leverage the solver
\[
\algname{FDS-FEGM} (V, \psi, \bfv, \lambda) = \FDS _{\Norm{\cdot}_{\cE}} \bigl(
    V, \psi, \bfv, \lambda
    \mid
    (\cM ^{\FEGM} _{i} ) _{i \in [K]}
\bigr) .
\]
Finally, we obtain the concrete algorithm \DMvip as follows:
\begin{mdframed}
\begin{equation}
\label{eq:final-implementation-of-DROMbcvip}
\begin{aligned}
\ROM _{\Norm{\cdot}_{\cE}} \Bigl(
    (V_i)_{i \in [K]}, (\psi_i)_{i \in [K]}, \bfz^0, (\lambda_{t})_{t \ge 1}
    \mid
    \algname{FDS-FEGM}
\Bigr) . %
\end{aligned}
\end{equation}
\end{mdframed}
Combining \cref{lemma:Reduced-gradient-method-for-VIPs,lemma:FDS-correctness-VIP,lemma:exact-accuracy-to-relative-accuracy-under-strong-monotonicity,lemma:making-the-residual-norm-small-full}, with the implementation in \cref{eq:final-implementation-of-DROMbcvip},
we conclude that \cref{thm:DROMbcvip-complexity} holds for the constant $C_0$ from \cref{lemma:making-the-residual-norm-small-full}. We include the complete proof below.

\begin{lemma}
\label{thm:DROMbcvip-complexity}
    Under \labelcref{item:bounded-distance-to-block-solution,item:operator-block-Lipschitz-continuity}, for
    \(
    \lambda_{t+1} \equiv \lambda \ge 2 \bar L_{\texttt{c}},
    \)
    \DMvip~(\cref{alg:template-decoupled-reduced-operator-method-for-block-VIP}) with the implementation in \cref{eq:final-implementation-of-DROMbcvip} takes no more than
    \[
    2T
    \]
    communication rounds and no more than
    \[
    T \cdot \bigl( 1 + C_0 \cdot \frac{3 L_{ii}} {\alpha_i \lambda} \bigr)
    \]
    queries to $V_i$, for all $i \in [K]$,
    and obtains an $\epsilon$-approximate solution
    \(
    \bar \bfz ^T ,
    \)
    where
    \[
    T = \Ceil[\big]{ \frac{ \sum _i {\alpha_i \lambda D_i^2} } { 2 \epsilon } }
    \]
    and $C_0 > 0$ is some fixed constant.
\end{lemma}

\begin{proof}[Proof of \cref{thm:DROMbcvip-complexity}]
    By \labelcref{item:monotone-operator}, we have
    \[
    \Delta (\bar \bfz^{T}) \le \Bigl( \sum _{t=0} ^{T-1} a_{t+1} \Bigr) ^{-1} \max _{\bfz \in \cB \cap Q}\ \biggl[ \sum _{t=0} ^{T-1} a_{t+1} \InnerProduct { V _\psi (\bfz^{t+1}) }{ \bfz^{t+1} - \bfz } \biggr] .
    \]
    Further, with $\lambda \ge 2 \bar L_{\texttt{c}}$, by \cref{lemma:Reduced-gradient-method-for-VIPs,lemma:FDS-correctness-VIP}, we have
    \[
    \begin{aligned}
    & \Delta (\bar \bfz^{T})
    \le \Bigl( \sum _{t=0} ^{T-1} a_{t+1} \Bigr) ^{-1} \max _{\bfz \in \cB \cap Q}\ \biggl[ \sum _{t=0} ^{T-1} a_{t+1} \InnerProduct { V _\psi (\bfz^{t+1}) }{ \bfz^{t+1} - \bfz } \biggr] \\
    &\quad \le \Bigl( \sum _{t=0} ^{T-1} a_{t+1} \Bigr) ^{-1} \Bigl[ \sum _{i \in [K]} \bigl( { \frac {\alpha_i} {2} \max _{\bfz_i \in \cB_i \cap \dom \psi_i}\ \Norm {\bfz^0_i - \bfz_i}_i ^2 } \bigr) \Bigr] \\
    &\quad \le \Bigl( \sum _{t=0} ^{T-1} \frac 1 {\lambda_{t+1}} \Bigr) ^{-1} \cdot \frac 1 2 \sum _{i \in [K]} \alpha_i D_i^2
    \le \epsilon ,
    \end{aligned}
    \]
    where the last inequality follows from the assignments of $(\lambda_{t}) _{t \ge 1}$ and $T$.
    Therefore, the number of communication rounds is bounded by $2T$.

    Now we count the number of gradient queries.
    By \cref{lemma:exact-accuracy-to-relative-accuracy-under-strong-monotonicity}, \FEGM always returns the solution with the required relative distance accuracy; and in view of \cref{lemma:making-the-residual-norm-small-full}, it takes no more than $C_0 \cdot { \frac{3 L_{ii} } {\alpha_i \lambda} } $ gradient queries to $V_i$, for all $i \in [K]$.
    Therefore, the numbers of queries to $V_i$ are bounded by $T \cdot \bigl( 1 + C_0 \cdot { \frac{3 L_{ii} } {\alpha_i \lambda} } \bigr)$, for all $i \in [K]$.
\end{proof}

\begin{remark}[Oracle comparison]
    Under the same choice of parameters as in \cref{thm:DROMbcvip-communication-cost}, the oracle cost of \DMvip is bounded by
    \begin{equation}
    \label{eq:computational-cost-block-VIPs}
    \frac 2 \epsilon \Bigl( \sum _{i \in [K]} c_i \Bigr) \Bigl( \sum _{i \in [K]} A_{i} \Bigr) + \frac {3 C_0} {\epsilon} \Bigl( \sum _{i \in [K]} \frac {B_i c_i} {A_i} \Bigr) \Bigl( \sum _{i \in [K]} A_{i} \Bigr) .
    \end{equation}
    Compared to the computational cost of \EG, which is given by
    \[
    \frac {1} \epsilon \Bigl( \sum _{i \in [K]} c_i \Bigr) \Bigl( \sum _{i \in [K]} A_{i} \Bigr) + \frac {1} \epsilon \Bigl( \sum _{i \in [K]} c_i \Bigr) \Bigl( \sum _{i \in [K]} B_{i} \Bigr) ,
    \]
    our computational cost in \cref{eq:computational-cost-block-VIPs} differs in the second term. Consequently, our \DMvip may offer an advantage or disadvantage depending on the relative conditioning of $A_i$, $B_i$, and $c_i$ for $i \in [K]$.
\end{remark}

\end{document}